\documentclass[11pt,letterpaper]{amsart}

\usepackage{amscd,amssymb,amsthm,verbatim}
\usepackage{enumerate}
\usepackage{bm}
\usepackage{mathrsfs,graphicx}
\usepackage{stmaryrd}
\usepackage{slashed}

\usepackage[dvipsnames]{xcolor}

\usepackage[
  bookmarks=true,
  colorlinks=true,
  linkcolor=blue,
  urlcolor=blue,
  citecolor=blue,
  plainpages=false,
  pdfpagelabels=true,
  final,
  breaklinks=true
]{hyperref}

\usepackage[numbers,sort&compress]{natbib}

\numberwithin{equation}{section}

\usepackage{enumitem}
\theoremstyle{plain}
\newtheorem{theorem}{Theorem}[section]
\newtheorem{proposition}[theorem]{Proposition}
\newtheorem{lemma}[theorem]{Lemma}
\newtheorem{corollary}[theorem]{Corollary}

\theoremstyle{definition}
\newtheorem{definition}[theorem]{Definition}
\theoremstyle{remark}
\newtheorem{remark}[theorem]{Remark}
\usepackage{csquotes}
\usepackage{graphicx}
\usepackage[dvipsnames]{xcolor}
\usepackage{float}

\colorlet{mycolor}{RedViolet}

\colorlet{mycolor2}{RedViolet}
\usepackage{accents}
\usepackage{pgfplots}
\pgfplotsset{compat=1.18}

\newcounter{Counter}
\theoremstyle{plain}

\newcommand{\tr}{\operatorname{tr}}

\newcommand{\Div}{\operatorname{div}}

\DeclareMathOperator{\Per}{Per}
\DeclareMathOperator{\graph}{graph}
\DeclareMathOperator{\spt}{spt}
\DeclareMathOperator{\Reg}{Reg}
\DeclareMathOperator{\Sing}{Sing}
\newcommand{\loc}{\mathrm{loc}}

\usepackage{MnSymbol}

\DeclareFontFamily{U}{MnSymbolC}{}
\DeclareSymbolFont{MnSyC}{U}{MnSymbolC}{m}{n}
\DeclareFontShape{U}{MnSymbolC}{m}{n}{
	<-6>  MnSymbolC5
	<6-7>  MnSymbolC6
	<7-8>  MnSymbolC7
	<8-9>  MnSymbolC8
	<9-10> MnSymbolC9
	<10-12> MnSymbolC10
	<12->   MnSymbolC12}{}
\DeclareMathSymbol{\intprod}{\mathbin}{MnSyC}{'270}

\title[Hyperboloidal and Spacetime PMT in all dimensions]
{The Hyperboloidal and Spacetime Positive Mass Theorem in All Dimensions}

\author[Hirsch]{Sven Hirsch}
\address{Columbia University, 2990 Broadway, New York, NY 10027, USA}
\email{sven.hirsch@columbia.edu}

\author[Khuri]{Marcus Khuri}
\address{Department of Mathematics, Stony Brook University, Stony Brook, NY, 11794, USA}
\email{marcus.khuri@stonybrook.edu}

\author[Lesourd]{Martin Lesourd}
\address{Sphere 28 LLC}
\email{mlesourd@sphere28.io}

\author[Zhang]{Yiyue Zhang}
\address{Beijing Institute of Mathematical Sciences and Applications, Beijing, 101408, China}
\email{zhangyiyue@bimsa.cn}

\begin{document}

\begin{abstract}
Using the recent work of Brendle--Wang on the Riemannian positive mass theorem, we prove the spacetime positive mass theorem for asymptotically flat and asymptotically hyperboloidal initial data sets in arbitrary dimensions.
\end{abstract}

\maketitle

\section{Introduction} \label{sec:intro}

In this work we prove the spacetime positive mass theorem for asymptotically flat and asymptotically hyperboloidal initial data sets satisfying the dominant energy condition, together with corresponding rigidity statements.
We also note the concurrent papers of Brendle-Wang \cite{BrendleWang2} and of Tsang \cite{Tsang} where similar results were obtained, as well as the approaches by Schoen-Yau \cite{SchoenYau22} and by Lohkamp \cite{Lohkamp1,Lohkamp2}.

\begin{theorem}\label{thm:main2}
Let $(M^n,g,k)$, $n\geq 3$ be a complete asymptotically hyperboloidal\footnote{Several results cited in this paper, such as \cite{HuangLee2, LundbergHyperbolic, BrendleWang, HirschHuang}, require slightly stronger decay than is usually assumed. To improve readability, we absorb this into our notion of asymptotically hyperboloidal and asymptotically flat initial data sets; cf.\ Remark \ref{remark stronger decay}.} initial data set satisfying the dominant energy condition
\[
\mu \ge |J|_g.
\]
Then its total energy--momentum vector $(E,P)$ satisfies
\[
E \ge |P|.
\]
Moreover, if $M^n$ is spin or $k=g$, then equality holds if and only if $(M^n,g,k)$ admits an isometric embedding as a spacelike hypersurface into Minkowski space.
\end{theorem}

We point out that the rigidity statement remains open in the non-spin case when $k\ne g$, and it would be interesting to close this gap; see also \cite{HuangLee, HuangLee2, HuangJangMartin, HirschHuang}.
We also note that there is an asymptotically AdS positive mass theorem \cite{ChruscielMaertenTod, Maerten,HirschZhang3}, which is so far completely unknown in the non-spin setting.

\begin{theorem}\label{thm:main}
Let $(M^n,g,k)$, $n\geq 3$ be a complete asymptotically flat initial data set satisfying the dominant energy condition
\[
\mu \ge |J|_g.
\]
Then its ADM energy--momentum vector $(E_{\mathrm{ADM}},P_{\mathrm{ADM}})$ satisfies
\[
E_{\mathrm{ADM}} \ge |P_{\mathrm{ADM}}|.
\]
Moreover, equality holds if and only if $(M^n,g,k)$ admits an isometric embedding as a spacelike hypersurface into a pp-wave spacetime
$(\mathbf M^{n+1},\mathbf g)$, where $\mathbf M^{n+1}=\mathbb R^{n+1}$ and
\[
\mathbf g
=
-2\,dt\,du
+
F\,du^2
+
dx_1^2+\cdots+dx_{n-1}^2,
\]
for a $t$-independent function $F$ satisfying
\[
\Delta_{\mathbb R^{n-1}}F(\cdot,u)\le 0
\qquad\text{for every }u\in\mathbb R.
\]
\end{theorem}

\begin{remark}
Planar waves with parallel rays, or pp-waves for short, are Lorentzian manifolds that model gravitational waves.
The superharmonicity of $F$ corresponds to the dominant energy condition $\mu\ge |J|_g$.
In dimensions $n=3$ and $n=4$, there are no nontrivial asymptotically flat pp-waves; that is, in this case one has $F=F(u)$ and $(\mathbf M^{n+1},\mathbf{g})$ is Minkowski space.
More generally, $(M^n,g,k)$ embeds in Minkowski space whenever $(M^n,g,k)$ is $C^{\ell,\alpha}_{-q}$-asymptotically flat with $q>n-1-\ell-\alpha$.
For a detailed overview, we refer to \cite{HirschZhang2} and the references therein.
On the other hand, as shown in \cite{HirschJangZhangHyperboloidal}, there are no asymptotically hyperboloidal analogues of pp-waves in any dimension.
\end{remark}
Both Theorem \ref{thm:main2} and Theorem \ref{thm:main} rely in an essential way on several deep previous results.

\medskip

\noindent To prove Theorem \ref{thm:main2}, we use:
\begin{enumerate}
    \item the existence and regularity theory for the Jang equation on asymptotically hyperboloidal manifolds due to Lundberg \cite{LundbergHyperbolic}, building on the pioneering work of Sakovich \cite{SakovichHyperbolic},
    
    \item the density theorem of Dahl--Sakovich \cite{DahlSakovich}, which is needed in order to apply the above Jang equation results,
    
    \item the all-dimensional Riemannian positive mass theorem of Brendle--Wang \cite{BrendleWang}, applied to the Jang graph,
    
    \item and the rigidity results of Hirsch--Jang--Zhang \cite{HirschJangZhangHyperboloidal} and Huang--Jang--Martin \cite{HuangJangMartin}.
\end{enumerate}

\medskip

\noindent To prove Theorem \ref{thm:main}, we use:
\begin{enumerate}
    \item the existence and regularity theory for the Jang equation on asymptotically flat manifolds due to Eichmair \cite{EichmairJang},
    
    \item the density theorem of Eichmair--Huang--Lee--Schoen \cite{EichmairHuangLeeSchoen}, which is needed in order to apply the above Jang equation theory; moreover, a related density result of \cite{EichmairHuangLeeSchoen} is crucial for applying the boost theorem of Christodoulou--\'O Murchadha \cite{ChristodoulouOMurchadha} in order to reduce the inequality \(E\ge |P|\) to the case \(E\ge 0\),
    
    \item the all-dimensional Riemannian positive mass theorem of Brendle--Wang \cite{BrendleWang}, again applied to the Jang graph,
    
    \item and the rigidity results of Hirsch--Huang \cite{HirschHuang}, Hirsch--Zhang \cite{HirschZhang2}, and Lee--Huang \cite{HuangLee, HuangLee2}.
\end{enumerate}

\medskip

The results listed above themselves rely on a large number of further important developments.  For a more detailed historical overview, we refer to Section \ref{S:history}.

\medskip

An important new observation is the following result, concerning the regularity of
singular Jang graphs. In partciular, we refine the known Hausdorff-codimension-seven 
estimate to a Minkowski-codimension-seven estimate for the singular set.
This is crucially needed in order to apply the methods of 
Brendle--Wang \cite{BrendleWang} within the context of the associated Jang manifold.

\begin{theorem}\label{thm:Jang}
Let $(M^n,g,k)$, $n\geq 3$ be either an asymptotically hyperboloidal or asymptotically flat initial data set, and let $\bar{\Sigma} \subset M^n \times \mathbb R$
be a geometric Jang surface obtained as a subsequential limit of solutions to the capillarity--regularized Jang equation.
Then $\bar{\Sigma}$ is the support of a $C$-almost minimizing boundary in $M^n \times \mathbb R$.
Furthermore, there is a closed singular set
$\Sing(\bar{\Sigma}) \subset \bar\Sigma$ such that $\bar{\Sigma} \setminus \Sing(\bar\Sigma)$ is smooth and its Minkowski dimension satisfies\footnote{We point out that, in order to prove Theorems \ref{thm:main2} and \ref{thm:main}, the weaker estimate $\dim_{\mathcal M} \Sing(\Sigma)< n-4$ would already suffice.}
\[
\dim_{\mathcal M} (\Sing(\bar\Sigma) )\le n-7.
\]
\end{theorem}

Another difficulty arises from the possibility that the Jang graph blows up at singular MOTS. In such a situation singularities can accumulate along the associated asymptotically cylindrical end, so that the resulting singular set is no longer bounded. We overcome this difficulty by constructing the conformal blow-up, piece by piece along any asymptotically cylindrical ends.

\medskip

Finally, using Theorem \ref{thm:main} together with gluing methods, we give an alternative and shorter proof of Theorem \ref{thm:main2} under additional assumptions.

\medskip

\noindent
\textbf{Acknowledgements:} SH thanks the Simons Foundation and the Banff International Research Center, where part of this work was carried out. MK acknowledges support from NSF Grant DMS-2405045. ML acknowledges the support of Sphere 28 LLC. YZ was partially supported by NSFC Grant No.\ 12501070 and the startup fund from BIMSA. The authors would like to thank Piotr Chru\'sciel and Lan-Hsuan Huang for helpful discussions and interest in this article.

\section{Historical Background}\label{S:history}

\subsection{Asymptotically flat initial data sets}

The spacetime positive mass theorem in the asymptotically flat setting has its
origins in the work of Schoen--Yau and Witten.  Schoen--Yau first proved the
Riemannian (\(k=0\)) case in dimension \(3\) using stable minimal surfaces
\cite{SchoenYau}.  They subsequently treated the general spacetime case in
dimension \(3\) by using Jang's equation to reduce the problem to the
Riemannian positive mass theorem \cite{Jang,SchoenYauII}.  Independently,
Witten introduced a spinorial proof \cite{Witten}, which was later placed on a
rigorous analytic foundation by Parker--Taubes \cite{ParkerTaubes} and
extended to higher-dimensional spin manifolds by Bartnik
\cite{Bartnik0}.

\medskip

The minimal hypersurface approach of Schoen--Yau was subsequently extended to
higher dimensions.  In the Riemannian case, Schoen--Yau proved the positive
mass theorem in dimensions \(3\le n\le 7\), where the relevant stable minimal
hypersurfaces are smooth \cite{SchoenYau79,SchoenYau22}.  
In the spacetime setting,
Eichmair extended the Jang-equation reduction to dimensions \(3\le n \le 7\)
\cite{EichmairJang}. 
A priori, the Jang equation \cite{SchoenYauII,EichmairJang} only yields the
inequality \(E_{ADM}\ge0\).
Eichmair--Huang--Lee--Schoen then gave a direct proof of
the spacetime positive mass theorem \(E_{ADM}\ge|P_{ADM}|\) in dimensions \(3\le n\leq 7\),
replacing minimal hypersurfaces by marginally outer trapped surfaces
\cite{EichmairHuangLeeSchoen}.  Their work also established an important
density theorem.  Together with the boost theorem of
Christodoulou--\'O Murchadha \cite{ChristodoulouOMurchadha}, this density
theorem allows one to reduce the inequality \(E_{ADM}\ge|P_{ADM}|\) to the case \(E_{ADM}\ge0\).

\medskip

In subsequent work, the dimension threshold \(7\) was improved for the
Riemannian positive mass theorem.
Foundational regularity results of Hardt--Simon and Smale addressed aspects of
the dimension \(8\) singularity problem \cite{HardtSimon,Generic8}.  More
recently, Chodosh--Mantoulidis--Schulze proved generic regularity results in
dimensions \(9\) and \(10\) \cite{Generic910}, and
Chodosh--Mantoulidis--Schulze--Wang extended these results to dimension \(11\)
\cite{Generic11}.  Building on conformal blow-up and spectral scalar curvature
ideas, Bi--Hao--He--Shi--Zhu proved the Riemannian positive mass theorem up to
dimension \(19\) \cite{BiHaoHeShiZhu}.  This line of work culminated in the
dimension-descent scheme of Brendle--Wang, which establishes the Riemannian
positive mass theorem in arbitrary dimension \cite{BrendleWang}.

\medskip

Several alternative approaches to the positive mass theorem have also been
developed.  
Hirsch--Kazaras--Khuri introduced spacetime harmonic functions and used them to
prove the spacetime positive mass theorem in dimension \(3\)
\cite{HirschKazarasKhuri}.
In the 3-dimensional Riemannian case, Huisken--Ilmanen obtained
the result as a consequence of weak inverse mean curvature flow and the
Riemannian Penrose inequality \cite{HuiskenIlmanen}.  Li gave a proof using
Ricci flow \cite{LiRicciFlowPMT}.  Bray--Kazaras--Khuri--Stern introduced a
harmonic-function approach based on level sets and the Bochner formula
\cite{BrayKazarasKhuriStern}, while Agostiniani--Mazzieri--Oronzio obtained a
proof using Green's functions and nonlinear potential-theoretic monotonicity
formulas \cite{AgostinianiMazzieriOronzio}.  

\medskip

There have also been many important extensions of the positive mass theorem.
These developments include the Penrose inequality
\cite{HuiskenIlmanen,BrayPenrose}, Penrose-type inequalities
\cite{HerzlichPenrose}, mass-capacity inequalities
\cite{BrayPenrose,AgostinianiMazzieriOronzio}, positive mass theorems with
corners and creases \cite{KazarasKhuriLin, MiaoPMT,ShiTam}, positive mass theorems with arbitrary ends
\cite{LesourdUngerYau}, extensions involving electric and magnetic fields
\cite{BartnikChrusciel,GHHP,GibbonsHull}, spin positive mass theorems under
dominant energy shields \cite{CecchiniLesourdZeidlerShields}, positive mass
theorems with angular momentum \cite{Dain,HKWX,SchoenZhou}, multiple time
dimensions \cite{HirschPayneZhang}, with boundary \cite{HirschMiaoBoundary, McCormickMiaoPenrose}, and low-regularity positive mass theorems
\cite{LeeLeFloch}.

\medskip

The rigidity theory in the asymptotically flat spacetime setting has an equally
rich history.  Witten \cite{Witten} already predicted the possible appearance
of pp-wave spacetimes, and early work on the equality case includes the study
by Yip \cite{Yip}.  Beig--Chru\'sciel \cite{BeigChrusciel} proved in dimension \(3\)
that if \(E_{ADM}=|P_{ADM}|\) and suitable additional decay
assumptions hold, then necessarily \(E_{ADM}=0\).
Moreover, they showed that if \(E_{ADM}=0\), then \((M^n,g,k)\) embeds into Minkowski
spacetime; this was also achieved by Schoen--Yau \cite{SchoenYauII} using the Jang equation.  
Chru\'sciel--Maerten extended their approach to higher-dimensional
spin manifolds in \cite{ChruscielMaerten}. 

\medskip

In a series of papers, Huang--Lee developed the Lagrange multiplier method for
the equality case \cite{HuangLee,HuangLee2,HuangLee3}, and proved rigidity in
the non-spin setting.  Their work builds upon Bartnik's variational
perspective on mass-minimizing initial data \cite{Bartnik} together with the
Beig--Chru\'sciel reduction from \(E_{ADM}=|P_{ADM}|\) to \(E_{ADM}=0\).
In \cite{HuangLee2}, Huang--Lee also discovered asymptotically flat pp-wave
spacetimes satisfying \(E_{ADM}=|P_{ADM}|\) but \(E_{ADM}\neq0\).  In particular, the
Beig--Chru\'sciel reduction result cannot hold in full generality.

\medskip

Subsequently, Hirsch--Zhang \cite{HirschZhang2} proved that spin initial data
sets satisfying \(E_{ADM}=|P_{ADM}|\), without any additional asymptotic assumptions,
embed into pp-wave spacetimes.  This builds upon their earlier work on
rigidity via spacetime harmonic functions \cite{HirschZhang}.  Finally,
combining Huang--Lee's Lagrange multiplier method
\cite{HuangLee,HuangLee2} with a new monotonicity formula for causal Killing
vector fields, Hirsch--Huang \cite{HirschHuang} showed that initial data sets
embed into pp-wave spacetimes also in the non-spin setting.
The rigidity problem in the charged setting remains open in full generality, with partial results due
to Chru\'sciel--Reall--Tod \cite{ChruscielReallTod}.

\subsection{Asymptotically hyperboloidal, asymptotically hyperbolic, and asymptotically AdS initial data sets}

The asymptotically hyperboloidal and asymptotically hyperbolic positive mass
theorems developed in parallel with the asymptotically flat theory.  In the
time-symmetric asymptotically hyperbolic setting, Wang
\cite{WangHyperbolic} and Chru\'sciel--Herzlich
\cite{ChruscielHerzlich} proved positive mass theorems by spinorial methods.
In dimension \(3\), Zhang introduced a definition of total energy--momentum
for asymptotically hyperbolic initial data and proved a corresponding positive
mass theorem \cite{XiaoZhang}.  The structure of conserved quantities in the
asymptotically hyperbolic setting was further developed by
Chen--Wang--Yau \cite{ChenWangYauHyperbolic}, and related hyperboloidal
mass quantities were studied by Chru\'sciel--Jezierski--\L{}\c{e}ski 
\cite{ChruscielJezierskiLeski}.

\medskip

The asymptotically anti-de Sitter case is closely related, but the conserved
quantities are organized by the asymptotic symmetry group of anti-de Sitter
space.  Spinorial positive mass inequalities in this setting were proved by
Maerten \cite{Maerten} and by Chru\'sciel--Maerten--Tod
\cite{ChruscielMaertenTod}, and the rigidity by Hirsch--Zhang \cite{HirschZhang3}.  These results remain fundamentally spinorial, and
the corresponding asymptotically AdS positive mass theorem is still not known
in the non-spin setting. There is also the
volume-renormalized mass introduced by Dahl--Kr\"oncke--McCormick
\cite{DahlKroenckeMcCormick}, whose positivity was subsequently established by Kr\"oncke--Oronzio--Pinoy
 \cite{Kroencke2}.

\medskip

The first non-spinorial asymptotically hyperbolic positive mass theorem was
proved by Andersson--Cai--Galloway in dimensions \(3\le n\le 7\), under a sign
condition on the mass aspect function \cite{AnderssonCaiGalloway}.
Chru\'sciel--Galloway later proved positive mass theorems for asymptotically
hyperbolic manifolds with boundary \cite{ChruscielGallowayBoundary}, and
Chru\'sciel--Delay \cite{ChruscielDelayHyperbolic} showed that the spacetime 
asymptotically flat positive mass theorem implies the asymptotically hyperbolic 
positive mass theorem in all dimensions.
Incarnations of the positive mass theorem have also been achieved in the asymptotically locally
hyperbolic setting with toroidal infinities by Chru\'sciel--Galloway--Nguyen--Paetz \cite{CGNP}, Alaee--Hung--Khuri \cite{AHK}, and Lee--Neves \cite{LN}, while the related Horowitz--Myers conjecture was recently resolved by Brendle--Hung \cite{BrendleHung,BrendleHung1}.


\medskip

For general hyperboloidal initial data, Schoen--Yau already sketched how the
Jang equation method should apply to the positivity of the Bondi mass
\cite{SchoenYauBondi}.  Sakovich carried out this program in dimension \(3\),
adapting the Schoen--Yau reduction to asymptotically hyperboloidal initial data
and obtaining a non-spinorial proof of the hyperboloidal positive mass theorem
\cite{SakovichHyperbolic}.  Lundberg \cite{LundbergHyperbolic} subsequently extended the Jang equation
analysis to higher-dimensional asymptotically hyperboloidal initial data in
dimensions less than 8 .  As in the asymptotically
flat case, the restriction to dimensions below eight reflects the regularity
theory available for the relevant geometric measure theory objects.

\medskip

The rigidity theory in the asymptotically hyperbolic and hyperboloidal settings
is similarly subtle.  In the Riemannian asymptotically hyperbolic case,
rigidity is part of the spinorial theorems of Wang and
Chru\'sciel--Herzlich \cite{WangHyperbolic,ChruscielHerzlich}, and also of the
non-spinorial theorem of Andersson--Cai--Galloway
\cite{AnderssonCaiGalloway}.  Huang--Jang--Martin developed further rigidity
and mass-minimization results in the asymptotically hyperbolic setting
\cite{HuangJangMartin}.  For general asymptotically hyperboloidal initial data
sets which are spin, Hirsch--Jang--Zhang showed that vanishing hyperboloidal mass forces the
data to arise from Minkowski space
\cite{HirschJangZhangHyperboloidal}; in particular, unlike in the
asymptotically flat case, there are no asymptotically hyperboloidal analogues
of pp-waves.

\section{Preliminaries}\label{sec: preliminaries}

\subsection{Asymptotically hyperboloidal initial data sets}

Let
\begin{equation}
b=\frac{dr^2}{1+r^2}+r^2g_{S^{n-1}}
\end{equation}
be the hyperbolic metric on $\mathbb H^n$ in hyperboloidal polar coordinates, where $g_{S^{n-1}}$ denotes the standard round metric on $S^{n-1}$.

\begin{definition}[Asymptotically hyperboloidal initial data sets]\label{Def:AH}
Let $(M^n,g)$, $n\geq 3$ be a connected, complete Riemannian manifold without boundary,
and let $k$ be a symmetric 2-tensor on $M^n$.
Fix $\ell\ge 6$, $\alpha\in(0,1)$, $\tau\in\bigl(\tfrac n2,n\bigr)$, and $\tau_0>0$.
We say that $(M^n,g,k)$ is an \emph{asymptotically hyperboloidal initial data set}
of type $(\ell,\alpha,\tau,\tau_0)$ if there exists a compact set
$\mathcal C\subset M^n$ and a diffeomorphism
\begin{equation}
\varphi:M^n_{\mathrm{end}}:=M^n\setminus \mathcal C\to \mathbb H^n\setminus \overline{B}
\end{equation}
such that in the corresponding asymptotic coordinates
\begin{equation}
(\varphi_*g-b,\varphi_*(k-g))
\in
C^{\ell,\alpha}_{-\tau}(\mathbb H^n\setminus \overline{B})
\times
C^{\ell-1,\alpha}_{-\tau}(\mathbb H^n\setminus \overline{B}),
\end{equation}
and
\begin{equation}
\varphi_*\mu,\;\varphi_*J\in C^{\ell-2,\alpha}_{-n-\tau_0}(\mathbb H^n\setminus \overline{B}),
\end{equation}
where $\overline{B}$ is the closure of a coordinate ball, and the energy density $\mu$ and momentum density $J$ are given by
\begin{equation}
\mu:=\frac12\Bigl(R_g+(\tr_g k)^2-|k|_g^2\Bigr),
\qquad
J:=\Div_g\bigl(k-(\tr_g k)g\bigr).
\end{equation}
\end{definition}

For the definitions of weighted H\"older space in the hyperbolic and Euclidean settings, see \cite[Definition 2.1]{HirschJangZhangHyperboloidal,HirschZhang2}.

\begin{definition}[Dominant energy condition]\label{Def:DEC}
We say that $(M^n,g,k)$ satisfies the \emph{dominant energy condition} if
\begin{equation}
\mu\ge |J|_g
\end{equation}
holds pointwise on $M^n$.
\end{definition}

The strict inequality upper bound for $\tau$ in Definition \ref{Def:AH} is included for certain density results, however, in some contexts it is important to allow $\tau=n$. In particular, we will slightly abuse terminology in the next definition due to this strength of decay.

\begin{definition}[Wang asymptotics]\label{Def:AH-Wang}
Let $(M^n,g,k)$ be asymptotically hyperboloidal of type $(\ell,\alpha,\tau=n,\tau_0)$.
We say that $(M^n,g,k)$ has \emph{Wang asymptotics} if in the asymptotic end
\begin{align}
\begin{split}
\varphi_*g-b &=\frac{\mathbf{m}}{r^{n-2}}+\mathcal{O}_{\ell,\alpha}(r^{-n-1}),\\
\left(\varphi_*k-b)\right|_{TS_r\times TS_r}
&=\frac{\mathbf{p}}{r^{n-2}}+\mathcal{O}_{\ell-1,\alpha}(r^{-n-1}),
\end{split}
\end{align}
where
\begin{equation}
\mathbf{m},\mathbf{p}\in C^{\ell,\alpha}\bigl(S^{n-1};\operatorname{Sym}^2(T^*S^{n-1})\bigr),
\end{equation}
and $\mathcal{O}_{\ell,\alpha}(r^{-n-1})$ represents a symmetric
2-tensor in the weighted space $C^{\ell,\alpha}_{n+1}(\mathbb{H}^n \setminus\overline{B})$ that vanishes in the radial direction. 
\end{definition}

Let us recall the notion of of total energy-momentum in the asymptotically hyperboloidal setting. Set
$\mathcal N:=\bigl\{V\in C^\infty(\mathbb H^n): \operatorname{Hess}^b V=Vb\bigr\}$ and observe that
\begin{equation}
\mathcal N=\operatorname{span}\{V^{(0)},V^{(1)},\dots,V^{(n)}\},
\end{equation}
where
$V^{(0)}=\sqrt{1+r^2}$, and $V^{(i)}=x^i$ for $i=1,\dots,n$. Here $x^i$ denote Cartesian coordinates on $\mathbb{R}^n$, and in this setting may be viewed as the restriction of the coordinates of Minkowski space restricted to the upper unit hyperboloid given by $t =\sqrt{1+r^2}$. The definition of total energy-momentum arises from a correspondence between the isometries of Minkowski space that preserve the hyperboloid, and functions in $\mathcal{N}$, see \cite[Section 2.2]{DahlSakovich}.
For asymptotically hyperboloidal initial data sets, the following notion of total energy and momentum is well-defined and is a geometric invariant \cite{ChruscielJezierskiLeski,Michel}.

\begin{definition}[Asymptotically hyperboloidal energy--momentum]\label{Def:AH-mass}
Given asymptotically hyperboloidal initial data $(M^n,g,k)$ with chart $\varphi$ at infinity, let
\begin{equation}
e:=\varphi_*g-b,
\qquad\quad
\eta:=\varphi_*(k-g),
\end{equation}
and consider the \textit{mass functional} $\mathcal{M}_{\varphi}:\mathcal{N}\rightarrow\mathbb{R}$ prescribed by
\begin{equation*}
\mathcal{M}_{\varphi}(V)
:=
\lim_{r\to\infty}
\int_{S_r}
\Big(
V(\Div_b e-d\tr_b e)
+\tr_b(e+2\eta)\,dV
-(e+2\eta)(\nabla^b V,\cdot)
\Big)(\nu^b)\,dA.
\end{equation*}
Here $S_r\subset \mathbb H^n$ is a coordinate sphere of radius $r$ having unit outer normal $\nu^b$ with respect to $b$, and $dA$ is the induced hypersurface measure.
The asymptotically hyperboloidal \textit{total energy--momentum} vector is then set as
\begin{equation}
E:=\frac{\mathcal{M}_{\varphi}(V^{(0)})}{2(n-1)\omega_{n-1}},
\qquad\quad
P^i:=\frac{\mathcal{M}_{\varphi}(V^{(i)})}{2(n-1)\omega_{n-1}},
\quad i=1,\dots,n,
\end{equation}
where $\omega_{n-1}$ denotes the volume of the unit $(n-1)$-sphere.
We also write $P=(P^1,\dots,P^n)$.
\end{definition}

It should be noted that in the case of Wang asymptotics, the total energy and momentum admit the following explicit formulas
\begin{equation}
E
=
\frac{1}{(n-1)\omega_{n-1}}
\int_{S^{n-1}}
\left(
\tr_{g_{S^{n-1}}}\mathbf{p}
+\frac{n-2}{2}\tr_{g_{S^{n-1}}}\mathbf{m}
\right)\,dA,
\end{equation}
and
\begin{equation}
P^i
=
\frac{1}{(n-1)\omega_{n-1}}
\int_{S^{n-1}}
\left(
\tr_{g_{S^{n-1}}}\mathbf{p}
+\frac{n-2}{2}\tr_{g_{S^{n-1}}}\mathbf{m}
\right)x^i\,dA,
\end{equation}
for $i=1,\dots,n$.

\subsection{Asymptotically flat initial data sets}\label{subsec:AF}

We recall the standard definition of asymptotically flat initial data sets and
their ADM energy--momentum.

\begin{definition}[Asymptotically flat initial data sets]\label{Def:AF}
Let $(M^n,g)$, $n\geq 3$ be a connected, complete Riemannian manifold without boundary,
and let $k$ be a symmetric 2-tensor on $M^n$.
Fix $\ell\ge 6$, $\alpha\in(0,1)$, $q\in\left(\frac{n-2}{2},\,n-2\right)$, and $q_0>0$.
We say that $(M^n,g,k)$ is an \emph{asymptotically flat initial data set}
of type $(\ell,\alpha, q, q_0)$ if there exists a compact set
$\mathcal C\subset M^n$ and a diffeomorphism
\begin{equation}
\varphi:M^n_{\mathrm{end}}:=M^n\setminus \mathcal C\to \mathbb R^n\setminus \overline{B}
\end{equation}
onto the complement of a Euclidean ball, such that in the corresponding
asymptotic Cartesian coordinates
\begin{equation}
\bigl(\varphi_*g-\delta,\varphi_*k\bigr)
\in
C^{\ell,\alpha}_{-q}(\mathbb R^n\setminus \overline{B})
\times
C^{\ell-1,\alpha}_{-q-1}(\mathbb R^n\setminus \overline{B}),
\end{equation}
and
\begin{equation}
\varphi_\ast \mu,\; \varphi_\ast J\in C^{0,\alpha}_{-n-q_0}(M^n),
\end{equation}
where $\delta$ denotes the Euclidean metric.
\end{definition}

The following notion of total energy and momentum for asymptotically flat initial data sets is well-defined and is a geometric invariant \cite{Bartnik0,Chrusciel}.

\begin{definition}[ADM energy and momentum]\label{Def:ADM}
Let $(M^n,g,k)$ be an asymptotically flat initial data set, and fix asymptotic coordinates
$x=(x^1,\dots,x^n)$ in the single end.
Let $\nu$ and $dA$ denote respectively the outward Euclidean unit normal and
Euclidean hypersurface measure on the coordinate sphere $S_r$.
Then the \emph{ADM energy} and \emph{ADM linear momentum} are given by
\begin{equation}
E_{\mathrm{ADM}}
:=
\frac{1}{2(n-1)\omega_{n-1}}
\lim_{r\to\infty}
\int_{S_r}\sum_{i,j=1}^n
(\partial_i g_{ij}-\partial_j g_{ii})\,\nu^j\,dA,
\end{equation}
and
\begin{equation}
P_{\mathrm{ADM}}^i
:=
\frac{1}{(n-1)\omega_{n-1}}
\lim_{r\to\infty}
\int_{S_r}\sum_{j=1}^n
\pi_j^i\,\nu^j\,dA,
\qquad i=1,\dots,n,
\end{equation}
where $\pi:=k-(\tr_g k)g$ is the conjugate momentum tensor.
We will write $P_{\mathrm{ADM}}=(P_{\mathrm{ADM}}^1,\dots,P_{\mathrm{ADM}}^n)$, and if $E_{\mathrm{ADM}}\ge |P_{\mathrm{ADM}}|$ then the corresponding \textit{ADM mass} is given by
\begin{equation}
m_{\mathrm{ADM}}:=\sqrt{E^2_{\mathrm{ADM}}-|P_{\mathrm{ADM}}|^2}.
\end{equation}
\end{definition}

\begin{remark}\label{remark stronger decay}
Typically, one assumes only $\ell\ge 2$ in the definitions of asymptotically
hyperboloidal and asymptotically flat initial data sets.
By imposing the stronger assumption $\ell\ge 6$, we are able to directly invoke
the results in \cite{HuangLee2, LundbergHyperbolic, BrendleWang, HirschHuang}.
Apart from this, the condition $\ell\ge 6$ is not used in the present work.
In addition, we impose stronger decay on $\mu$ and $J$ beyond being in $L^1(M^n)$, which is
required for \cite{EichmairHuangLeeSchoen,LundbergHyperbolic}.
\end{remark}

\section{Reduction to $E\ge 0$}\label{sec:m->E reduction}

\subsection{The asymptotically hyperboloidal case}

We show that, in the asymptotically hyperboloidal setting, it suffices to prove
nonnegativity of the energy in an arbitrary asymptotic chart. This is based on
a classical boost argument, see for instance \cite[Section~3.3]{HirschZhang3}.

\medskip

Let $\varphi:M_{\mathrm{end}}^n\to \mathbb H^n\setminus \overline{B}$ be the chosen
asymptotic chart, and let
\begin{equation}
(E,P)=(E,P^1,\dots,P^n)
\end{equation}
be the corresponding energy--momentum vector from
Definition~\ref{Def:AH-mass}. Recall that in hyperboloidal coordinates on
$\mathbb H^n\subset \mathbb R^{1,n}$ we have
\begin{equation}
t=\sqrt{1+r^2},
\qquad
x=(x^1,\dots,x^n),
\end{equation}
and
\begin{equation}
    E=\frac{\mathcal{M}_{\varphi}(t)}{2(n-1)\omega_{n-1}},
\qquad
P^i=\frac{\mathcal{M}_{\varphi}(x^i)}{2(n-1)\omega_{n-1}},
\end{equation}
where $\mathcal{M}_{\varphi}$ denotes the mass functional associated with $\varphi$.

\begin{proposition}\label{prop:AH-boost}
Let $\Phi\in SO(1,n)$ be a hyperbolic isometry, and define a new asymptotic
chart by
\[
\mathring\varphi:=\Phi\circ \varphi.
\]
Then the new energy--momentum vector satisfies
\[
(\mathring E,\mathring P)=\Phi(E,P).
\]
In particular, after a spatial rotation we may assume that
\[
P^1=|P|,
\qquad
P^i=0,\quad i=2,\dots,n.
\]
If additionally $E>|P|$, then the boost
\[
\Phi=
\begin{pmatrix}
\cosh\theta & \sinh\theta & 0 \\
\sinh\theta & \cosh\theta & 0 \\
0 & 0 & I_{n-1}
\end{pmatrix},
\quad
\cosh\theta=\frac{E}{\sqrt{E^2-|P|^2}},
\quad
\sinh\theta=-\frac{|P|}{\sqrt{E^2-|P|^2}},
\]
yields
\[
\mathring E=\sqrt{E^2-|P|^2},
\qquad
\mathring P=0.
\]
\end{proposition}

\begin{proof}
Since $\Phi$ is an isometry of $(\mathbb H^n,b)$, the composition
$\mathring\varphi=\Phi\circ\varphi$ is again an admissible asymptotic chart. Furthermore
if $V\in\mathcal N$, then the mass
functional computed in the new chart satisfies
\begin{equation}
    \mathcal{M}_{\mathring\varphi}(V)=\mathcal{M}_{\varphi}(V\circ\Phi).
\end{equation}
Indeed, the integrands in the flux formula are preserved by the isometry
$\Phi$, and hence the corresponding limits agree. Now let $\mathring t:=t\circ\Phi$ and
$\mathring x^i:=x^i\circ\Phi$, and observe that
\begin{align}
\begin{split}
    \mathring E &=\frac{\mathcal{M}_{\mathring\varphi}(t)}{2(n-1)\omega_{n-1}}
=\frac{\mathcal{M}_{\varphi}(\mathring t)}{2(n-1)\omega_{n-1}},\\
\mathring P^i &=\frac{\mathcal{M}_{\mathring\varphi}(x^i)}{2(n-1)\omega_{n-1}}
=\frac{\mathcal{M}_{\varphi}(\mathring x^i)}{2(n-1)\omega_{n-1}}.
\end{split}
\end{align}
Since $\mathcal{M}_{\varphi}$ is linear on $\mathcal N$, this shows that the
energy--momentum vector transforms by the Lorentz transformation $\Phi$.
For the explicit boost, we compute
\begin{equation}
    \mathring t
=
t\cosh\theta+x^1\sinh\theta
=
t\,\frac{E}{\sqrt{E^2-|P|^2}}
-
x^1\,\frac{|P|}{\sqrt{E^2-|P|^2}},
\end{equation}
\begin{equation}
    \mathring x^1
=
x^1\cosh\theta+t\sinh\theta
=
x^1\,\frac{E}{\sqrt{E^2-|P|^2}}
-
t\,\frac{|P|}{\sqrt{E^2-|P|^2}},
\end{equation}
and $\mathring x^i=x^i$ when $i=2,\dots,n$. Therefore
\begin{equation}
    \mathring E
=
\frac{\mathcal{M}_{\varphi}(\mathring t)}{2(n-1)\omega_{n-1}}
=
\frac{E^2-|P|^2}{\sqrt{E^2-|P|^2}}
=
\sqrt{E^2-|P|^2},
\end{equation}
while
\begin{equation}
    \mathring P^1
=
\frac{\mathcal{M}_{\varphi}(\mathring x^1)}{2(n-1)\omega_{n-1}}
=
\frac{E|P|-E|P|}{\sqrt{E^2-|P|^2}}
=
0,
\end{equation}
and clearly $\mathring P^i=0$ for $i=2,\dots,n$.
\end{proof}

As a consequence of Proposition \ref{prop:AH-boost}, we find that in order to prove the inequality portion of Theorem~\ref{thm:main2}, it is enough to establish nonnegativity of the energy in every admissible asymptotic chart.

\begin{corollary}\label{cor:AH-reduction}
Given an asymptotically hyperboloidal initial data set satisfying
the dominant energy condition, assume that for every admissible asymptotic chart the corresponding total energy is nonnegative. Then the full energy--momentum inequality
\[
E\ge |P|
\]
holds.
\end{corollary}

\begin{proof}
Suppose, by way of contradiction, that $(E,P)$ is not future causal, that is $E<|P|$. Then there exists a future unit timelike vector
\begin{equation}
a=(a_0,a_1,\dots,a_n)\in \mathbb R^{1,n},
\qquad
a_0>0,
\qquad
a_0^2-\sum_{i=1}^n a_i^2=1,
\end{equation}
such that
\begin{equation}
a_0E+\sum_{i=1}^n a_iP^i<0.
\end{equation}
Choose $\Phi\in SO(1,n)$ whose first row is $a$, and let $\mathring\varphi=\Phi\circ\varphi$.
By Proposition~\ref{prop:AH-boost}, the corresponding energy satisfies
\begin{equation}
    \mathring E=a_0E+\sum_{i=1}^n a_iP^i<0,
\end{equation}
contradicting the assumed nonnegativity of the energy in every chart. Hence $E\ge |P|$.
\end{proof}

An important component in the proof of Theorem~\ref{thm:main2} is the following density result
from Dahl-Sakovich \cite{DahlSakovich}, concerning perturbations to Wang asymptotics and a strict dominant
energy condition with controlled fall-off.

\begin{theorem} \label{DS density}
    Let $(M^n,g,k)$ be an asymptotically hyperboloidal initial data set of type $(\ell,\alpha,\tau,\tau_0)$ satisfying the dominant energy condition. Then for every $\varepsilon>0$ there exists an initial data set $(M^n,g',k')$ of type $(\ell-1,\alpha,n,\tau'_0)$ with Wang asymptotics, a strict dominant energy condition $\mu'>|J'|_{g'}$,
     and such that its energy--momentum vector satisfies
    \begin{equation*}
        |E'-E|+|P'-P|< \varepsilon.
    \end{equation*}
    Moreover, there exists $\lambda>0$ such that
    \begin{equation*}
        \mu'-|J'|_{g'}\geq\lambda r^{-n-1}\quad\text{ on }M^n_{\mathrm{end}}.
    \end{equation*} 
\end{theorem}

\begin{proof}
The proof closely follows the arguments in \cite[Theorems 3.1 \& 5.2]{DahlSakovich}.
By \cite[Theorem 5.3]{DahlSakovich}, after a perturbation, we may assume that $(M^n,g,k)$ is an initial data set of type $(\ell-1,\alpha,n,\tau_0')$ with Wang asymptotics, and $\mu>|J|_g$.
Choose a bounded positive function $\mathfrak w$ such that $\mathfrak w=r^{-n-1}$ near infinity. Then there exists $(g', k')$ satisfying, by \cite[equations (23) and (24)]{DahlSakovich}, the following properties: the new energy--momentum is $\varepsilon$-close to that of $(g,k)$, and
\begin{equation}
(1+tv)^{\frac{4}{n-2}}\mu'>\mu+\frac{t}{3}\mathfrak w,\qquad (1+tv)^{\frac{4}{n-2}}|J'|_{g'}<|J|_g+\frac{t}{4}\mathfrak w,
\end{equation}
where $v\in C^{\ell-1,\alpha}_{-n}$ 
and $t>0$ is sufficiently small. Moreover, according to \cite[Theorem 5.2]{DahlSakovich}, there exists a coordinate chart such that $(g',k')$ has Wang asymptotics. 
\end{proof}

\subsection{The asymptotically flat case}

Similarly to the asymptotically hyperboloidal case, a perturbation to strict dominant energy condition
with controlled fall-off, and harmonic asymptotics, will play an important role in the proof of Theorem \ref{thm:main2}.
Let $\pi:=k-(\tr_g k)\,g$, and recall that an asymptotically flat initial data set $(M^n,g,k)$ has \textit{harmonic asymptotics} on $M_{\mathrm{end}}^n$ if 
\begin{equation}
g_{ij}=u^{\frac{4}{n-2}}\delta_{ij},
\qquad
\pi_{ij}
=
u^{\frac{2}{n-2}}
\Big(
(L_\delta Y)_{ij}
-
(\operatorname{div}_\delta Y)\delta_{ij}
\Big),
\end{equation}
for some function $u$ and vector field $Y$ satisfying
\begin{equation}
u(x)=1+a|x|^{2-n}+O_{2,\alpha}(|x|^{1-n}),
\qquad
Y_i(x)=b_i|x|^{2-n}+O_{2,\alpha}(|x|^{1-n}),
\end{equation}
where $a,b_1,\dots,b_n$ are constants.
The next result is due to Eichmair-Huang-Lee-Schoen \cite{EichmairHuangLeeSchoen}.
We include some details here for completeness.

\begin{theorem}\label{thm:reduction}
Suppose that there exists an asymptotically flat initial data set $(M^n,g,k)$
satisfying the dominant energy condition with
\[
E_{\mathrm{ADM}}<|P_{\mathrm{ADM}}|.
\]
Then there exists another asymptotically flat initial data set, which by abuse
of notation we still denote by $(M^n,g,k)$, with harmonic asymptotics, a strict dominant energy condition $\mu>|J|_g$,
and such that its energy satisfies
\[
E_{\mathrm{ADM}}<0.
\]
Moreover, there exists $\lambda>0$ such that
\begin{equation*}
    \mu-|J|_{g}\geq\lambda r^{-n-1}\quad\text{ on }M^n_{\mathrm{end}}.
\end{equation*} 
\end{theorem}

\begin{proof}
Let $(M^n,g,k)$ be an asymptotically flat initial data set 
satisfying $\mu\ge |J|_g$ and $E_{\mathrm{ADM}}<|P_{\mathrm{ADM}}|$.
We first use the density theorem of Eichmair--Huang--Lee--Schoen \cite[page 119]{EichmairHuangLeeSchoen}, together
with the remark following its proof, to perturb the data slightly so that the
dominant energy condition is preserved, the inequality $E_{\mathrm{ADM}}<|P_{\mathrm{ADM}}|$ still holds,
and
\begin{equation}
\mu=0,
\qquad
J=0
\end{equation}
outside a large compact set. 
After this perturbation, we continue to denote the resulting initial data set by
$(M^n,g,k)$.

\medskip

If already $E_{\mathrm{ADM}}<0$, there is nothing further to prove at this stage. Otherwise
we have
\begin{equation}
0\le E_{\mathrm{ADM}}<|P_{\mathrm{ADM}}|.
\end{equation}
Choose coordinates so that $P_{\mathrm{ADM}}$ points in the $x^n$-direction. Since the
energy--momentum vector is spacelike, we may choose a boost parameter
$\theta\in(0,1)$ with $\theta>E_{\mathrm{ADM}}/|P_{\mathrm{ADM}}|$. The Lorentz--transformed energy then
satisfies
\begin{equation}
E_{\mathrm{ADM}}^\theta=\frac{E_{\mathrm{ADM}}-\theta |P_{\mathrm{ADM}}|}{\sqrt{1-\theta^2}}<0.
\end{equation}
By the boost theorem of Christodoulou--\'O Murchadha \cite{ChristodoulouOMurchadha}, the vacuum end of the
spacetime development may be replaced by a boosted asymptotically flat slice
with energy $E_{\mathrm{ADM}}^\theta$. Thus we obtain a new asymptotically flat initial data
set, again denoted by $(M^n,g,k)$, satisfying
\begin{equation}
E_{\mathrm{ADM}}<0
\end{equation}
and still obeying the dominant energy condition.

\medskip

Finally, we apply the density theorem of
Eichmair--Huang--Lee--Schoen once more \cite[Theorem 18]{EichmairHuangLeeSchoen}. By choosing the perturbation sufficiently
small, the inequality $E_{\mathrm{ADM}}<0$ may be preserved, while producing an initial data set
with harmonic asymptotics, and a strict dominant energy condition with the additional control
\begin{equation}
\mu-|J|_g \geq\lambda r^{-n-1}\quad\text{ on }M^n_{\mathrm{end}},
\end{equation}
for some constant $\lambda>0$. This refined inequality can be obtained by choosing $\bar{\mu}=\mu+\lambda\eta(r)r^{-n-1}$ and $\bar{J}=J$ in \cite[Lemma 23]{EichmairHuangLeeSchoen}, where $\lambda$ is chosen sufficiently small and $\eta(r)$ is a cutoff function such that $\eta(r)=0$ on $B_{r_0}$, and $\eta(r)=1$ for $r\ge 2r_0$, with $0\le \eta(r)\le 1$. In partciular, there exists a perturbation of $(g,k)$ with energy--momentum density $(\bar{\mu},\bar{J})$. After applying the perturbation, we obtain a metric with harmonic asymptotics by \cite[Proposition 24]{EichmairHuangLeeSchoen}.
Relabeling this final perturbation by $(M^n,g,k)$ completes the proof.
\end{proof}

\section{Existence and Regularity Theory of Singular Jang Graphs} \label{sec:Jang}

In 1978 P. S. Jang introduced a quasilinear elliptic equation \cite{Jang}, which Schoen and Yau \cite{SchoenYauII} successfully employed in dimension $n=3$ to reduce the positive mass theorem for general initial data to the case of time symmetry. This was later extended to dimensions $3\leq n\leq 7$ by Eichmair \cite{EichmairJang}, who introduced a key advancement in the modern Jang equation approach with the use of geometric measure theory techniques to analyze weak  solutions. Analogs of these results in the asymptotically hyperboloidal setting were obtained by Sakovich \cite{SakovichHyperbolic} for dimension $n=3$, and more recently by Lundberg \cite{LundbergHyperbolic} for dimensions $3\leq n\leq 7$.

\medskip

Consider the regularized Jang equation
\begin{equation}\label{eq:reg-jang}
H(f_{s})-\mathrm{tr}_g(k)(f_{s})=s f_{s} \quad\quad\quad \text{ on }M^n.
\end{equation}
For each $s >0$ a solution $f_{s}\in C^{3,\alpha}_{loc}(M^n)$, $\alpha\in(0,1)$ may be obtained through an exhaustion procedure,
which entails solving an appropriate Dirichlet problem on finite exhausting domains and using barriers in the asymptotic ends to control decay. The resulting solutions \cite[Proposition 5]{EichmairJang}, \cite[Proposition 4.4]{LundbergHyperbolic} satisfy 
\begin{equation}
|f_s| \leq C s^{-1}\quad\quad\quad \text{ on all of } M^n
\end{equation}
via a maximum principle, where $C$ is a constant depending only on the initial data, which in the asymptotically flat setting may be taken to be $C=1+n\sup_{M^n}|k|_{g}$. 
Furthermore, in the asymptotically flat case the regularized solutions satisfy
\begin{equation}\label{b1}
|f_s(x)| \leq c_\beta |x|^{2-\beta} \quad\quad\quad \text{ for all } |x| \geq \Lambda_\beta, 
\end{equation}
where $\beta=1+q\in(2,n)$ when $n\geq 4$, and the constants $c_\beta > 0$ and $\Lambda_\beta \geq 1$ are independent of $\tau$. While in the asymptotically hyperboloidal case \cite[Proposition 4.4]{LundbergHyperbolic}
\begin{equation}\label{b2}
f_-(r) \leq f_{s}(x) \leq f_+(r) \quad\quad\quad \text{ for all } r\geq \Lambda,
\end{equation}
where $\Lambda\geq 1$ and the radial barriers $f_{\pm}$ are independent of $\tau$ with
\begin{equation}
f_{\pm}(r)=\sqrt{1+r^2}+O(r^{3-n})
\end{equation}
when $n\geq 4$.

\begin{theorem}[Higher-dimensional AF geometric Jang limit]\label{thm:higher-dim-jang-limit}
For $n\geq 4$, let $(M^{n},g,k)$ be a complete asymptotically flat initial data set of type $(\ell,\alpha,q,q_0)$.
Then there exists a sequence $s_j\downarrow 0$, a Caccioppoli set
$\mathbf{E}\subset M^n\times\mathbb R$, and a closed set
\[
\bar\Sigma:=\spt(\partial \mathbf{E})\subset M^n\times\mathbb R
\]
such that the following hold.

\begin{enumerate}
\item[\rm (i)]
$\partial \mathbf{E}$ is a $2C$-almost minimizing boundary in $M^n\times\mathbb R$.

\item[\rm (ii)]
If
\[
\Reg(\bar\Sigma):=\bar\Sigma\setminus\Sing(\bar\Sigma),
\]
then $\Reg(\bar\Sigma)$ is a $C^{3,\alpha}_{\loc}$ embedded hypersurface, and
\[
\dim_{\mathcal H}\bigl(\Sing(\bar\Sigma)\bigr)\le n-7.
\]

\item[\rm (iii)]
Every connected component $\Sigma$ of $\Reg(\bar\Sigma)$ is either
\begin{enumerate}
\item[\rm (a)]
a vertical cylinder
\[
\Sigma=\Sigma_{0}^{\Reg}\times\mathbb R,
\]
where $\Sigma_{0}^{\Reg}\subset M^n$ is a smooth embedded marginally trapped hypersurface, 
whose closure is a $2C$-almost minimizing boundary in $M^n$ with singular set of Hausdorff dimension at most $n-8$, or

\item[\rm (b)]
a graph
\[
\Sigma=\graph(f_{\Sigma},U_{\Sigma})
\]
over an open set $U_{\Sigma}\subset M^n$, where
$f_{\Sigma}\in C^{3,\alpha}_{\loc}(U_{\Sigma})$ solves the Jang equation
\[
H(f_{\Sigma})-\tr_g(k)(f_{\Sigma})=0
\qquad\text{on }U_{\Sigma}.
\]
\end{enumerate}

\item[\rm (iv)]
There exists a graphical component
\[
\Sigma_{\infty}=\graph(f_{\infty},U_{\infty})
\]
with induced metric $\bar{g}=g+df_{\infty}^2$
such that, for some $\Lambda_{\beta}\ge 1$ and $c_{\beta}>0$, we have $\{x\in M:|x|>\Lambda_{\beta}\}\subset U_{\infty}$ and
\[
|\nabla^l f_{\infty}(x)|\le c_{\beta}|x|^{2-\beta-l}
\quad\text{for }|x|>\Lambda_{\beta},
\]
with $\beta=1+q\in(2,n)$ and $l=0,\dots,\ell-1$.
In particular, the Jang graph is smooth near infinity, with an asymptotically flat end of type $(\ell-2,\alpha,q,q_0)$ which preserves the ADM energy.

\item[\rm (v)]
For every graphical component $\graph(f_{\Sigma},U_{\Sigma})$, the frontier
$\partial U_{\Sigma}$ is carried by cylindrical components of $\bar\Sigma$.
Along the regular part of the corresponding cross-sections one has
\[
f_{\Sigma}(x)\to \pm\infty,
\]
and vertical translates of the graph converge in the sense of currents to the
corresponding cylinders.
\end{enumerate}
\end{theorem}

\begin{theorem}[Higher-dimensional AH geometric Jang limit]
\label{thm:higher-dim-jang-limit-lundberg}
For $n\geq 4$, let \((M^{n},g,k)\) be a complete 
asymptotically hyperboloidal initial data set of type $(\ell,\alpha,n,\tau_0)$ with Wang asymptotics.
Then the conclusions
\({\rm (i)}\), \({\rm (ii)}\), \({\rm (iii)}\), and \({\rm (v)}\) of
Theorem~\ref{thm:higher-dim-jang-limit} hold unchanged. Moreover, the conclusion
\({\rm (iv)}\) is replaced by the following statement.

\begin{enumerate}
\item[\rm (iv)]
There exists a graphical component
\[
        \Sigma_{\infty}=\graph(f_{\infty},U_{\infty})
\]
which contains the chosen asymptotic end. Moreover, for any $\epsilon>0$ there exists \(\Lambda\geq 1\)
such that $\{r>\Lambda\}\subset U_{\infty}$ and
\[
        f_{\infty}(r,\theta)
        =
        \sqrt{1+r^2}
        +
        \frac{\boldsymbol{\alpha}(\theta)}{r^{n-3}}
        +
        O_5(r^{2-n+\epsilon}),
\]
where $\boldsymbol{\alpha}\in C^{2}(S^{n-1})$ is the unique solution of
\[
        \Delta_{g_{S^{n-1}}}\boldsymbol{\alpha}-(n-3)\boldsymbol{\alpha}
        =
        \frac{n-2}{2}\operatorname{tr}_{g_{S^{n-1}}}\mathbf m
        +
        \operatorname{tr}_{g_{S^{n-1}}}\mathbf p.
\]
In particular, the Jang graph is smooth near infinity and its end is asymptotically flat of type $(\ell-2,\alpha, n-2,0)$,
without the scalar curvature integrability condition.
Moreover, the ADM energy of the Jang graph is related to that of the initial data by $\bar{E}_{ADM}=(n-1)E$.
\end{enumerate}
\end{theorem}

\begin{remark}
We note that, strictly speaking, use of the phrase `asymptotically flat' for the Jang graph in (iv) of Theorem \ref{thm:higher-dim-jang-limit-lundberg} is a slight abuse of terminology based on Definition \ref{Def:AF}, as the ranges of parameters and the integrability of scalar curvature are not enforced. Furthermore, since the Jang scalar curvature is not necessarily integrable in this situation, it is necessary to specify the asymptotically flat chart that is being used when discussing the ADM energy in this case, namely that employed in \cite[Proposition C.1]{LundbergHyperbolic}.
\end{remark}

\begin{proof}[Proof of Theorems \ref{thm:higher-dim-jang-limit} and \ref{thm:higher-dim-jang-limit-lundberg}]
These theorems are a combination of existing results in the literature. Here we will summarize the main points 
and provide the appropriate references. For each $s>0$, \cite[Proposition~6]{EichmairJang} shows that
$\Gamma_{s}:=\graph(f_{s})\subset M^n\times\mathbb R$ is a $2C$-almost minimizing boundary.
Choose a sequence $s_j\downarrow 0$. Since the $\Gamma_{s_j}$ are
uniformly $2C$-almost minimizing, the compactness theorem for
$2C$-almost minimizing boundaries yields, after passing to a subsequence,
a Caccioppoli set $\mathbf{E}\subset M^n\times\mathbb R$ such that
$\Gamma_{s_j}\to \partial \mathbf{E}$ as currents and as varifolds. It follows that
$\bar\Sigma:=\spt(\partial \mathbf E)$ is an $n$-dimensional almost minimizing hypersurface in the
ambient $(n+1)$-manifold $M^n\times\mathbb R$. Therefore, the regularity theorem
for almost minimizing boundaries applies to yield
$\dim_{\mathcal H}\bigl(\Sing(\bar\Sigma)\bigr)\le n-7$.
On $\Reg(\bar\Sigma)$, Allard regularity yields local $C^{1,\alpha}$ convergence
of a subsequence of the $\Gamma_{s_j}$, and the equation
\eqref{eq:reg-jang} upgrades this to local $C^{3,\alpha}$ convergence by
standard quasilinear elliptic estimates. This proves~(ii).

\medskip

Let $\Sigma$ be a connected component of $\Reg(\bar\Sigma)$. Since the
approximating graphs satisfy
\begin{equation}
H(f_{s_j})-\tr_g(k)(f_{s_j})=s_j f_{s_j},
\end{equation}
and since the convergence on compact subsets of $\Sigma$ is smooth while
$s_j f_{s_j}\to 0$ locally along the convergent sheets, one obtains
\begin{equation}
H_{\Sigma}-\tr_{\Sigma}(k)=0.
\end{equation} 
Let $v_{s}=\sqrt{1+|\nabla f_{s}|_g^2}$, and note that this quantity \cite[Lemma A.1]{EichmairMetzger} (see also \cite[Lemma 2.3]{EichmairPlateau}) satisfies
\begin{equation}
\Delta_{\Sigma_{s}}v_{s}^{-1}=-(|h_{s}|^2 +\nu_{s}(H_{s})+\mathrm{Ric}(\nu_{s},\nu_{s}))v_{s}^{-1}\leq \sigma v_{s}^{-1} +\langle Z,\nabla_{\Sigma_{s}} v_{s}^{-1}\rangle,
\end{equation}
for some locally bounded function $\sigma$ and vector field $Z$ whose bounds depend only on the initial data, where $\mathrm{Ric}$ is the ambient Ricci curvature of $M^n \times \mathbb{R}$, $h_{s}$ is the second fundamental form of the graph, and $\nu_{s}=v_{s}^{-1}(f_{s}^{i}\partial_i -\partial_t)$ is the unit normal to the graphs. 
In particular, the Harnack principle for limits of graphs implies that every
connected component of the regular set of such a geometric limit is either a
vertical cylinder or a graph over an open subset of $M^n$. 
Hence each connected
component of $\Reg(\bar\Sigma)$ is one of the two types in~(iii), and in the
graphical case the defining function solves the unregularized Jang equation. See Eichmair \cite[Section 3]{EichmairPlateau}
for more details.

\medskip

Statement (iv) in both the asymptotically flat and asymptotically hyperboloidal cases follows from 
the interior gradient estimate \cite[Lemma 2.1]{EichmairPlateau} (\cite[Lemma 6.1]{LundbergHyperbolic})
together with a standard boot-strap, and the barrier bounds \eqref{b1} and \eqref{b2}. Moreover, the relation
between the ADM energy of the Jang graph and original energy of the asymptotically hyperboloidal
initial data is given by \cite[(A.3) and (C.1)]{LundbergHyperbolic}, namely $E_{ADM}=(n-1)E$.

\medskip

Lastly, the boundary behavior in~(v) is the standard blow-up mechanism for
geometric limits of regularized Jang graphs: near the frontier of a graphical
domain, vertical translates of the graph subconverge to cylindrical components,
and along the regular part of the corresponding cross-sections the defining
function tends uniformly to $+\infty$ or $-\infty$. This is exactly the same
cylinder-versus-graph alternative and translation argument used in the smooth
case, and it is local on the regular set, so it carries over unchanged here.
\end{proof}

A particularly advantageous feature of Jang graphs is the weak positivity property enjoyed by
their scalar curvature when the dominant energy condition is satisfied. This is recorded in
the next result \cite[(2.25)]{SchoenYauII}.

\begin{proposition}[Schoen--Yau scalar curvature identity]\label{SY scalar curvature id}
Let $\Sigma\subset\mathrm{Reg}(\bar{\Sigma})$ be a graphical connected component with
induced Jang metric $\bar{g}=g+df^2$. Then the scalar curvature of the Jang metric
takes the form
\[
        R_{\bar{g}}
        =
        2(\mu-J(\mathrm{w}))
        +
        |h-k|_{\bar g}^2
        +
        2|X|_{\bar g}^2
        -
        2\operatorname{div}_{\bar g} X ,
\]
where $h$ is the second fundamental form of $\Sigma$ and $\mathrm{w}$, $X$ are 1-forms given by
\[
\mathrm{w}_i=\frac{f_i}{\sqrt{1+|\nabla f|_g^2}},\quad\quad\quad h_{ij}=\frac{\nabla_{ij} f}{\sqrt{1+|\nabla f|_g^2}},\quad\quad\quad
X_i =\bar g^{jl}(h_{ij}-k_{ij})\mathrm{w}_l, 
\]
with $\nabla_{ij}$ denoting covariant differentiation with respect to $g$.
\end{proposition}

\begin{proof}[Proof of Theorem \ref{thm:Jang}]
As discussed in the proof of Theorems \ref{thm:higher-dim-jang-limit} and \ref{thm:higher-dim-jang-limit-lundberg},
a limiting geometric Jang surface $\bar{\Sigma}$ is known to be the support of a $C$-almost minimizing boundary
$\partial\mathbf{E}$ in $M^n\times\mathbb{R}$, see Eichmair \cite[Lemma A.2]{EichmairPlateau}.
Thus, it remains only to justify the Minkowski dimension estimate for the singular
set. Our argument is purely local, and does not distinguish between the AH and the AF cases.

\medskip

Let us recall the precise statement of the $C$-almost minimizing property.
If $W \Subset M^n\times \mathbb R$ is open and $\mathbf X$ is an integral
$(n+1)$-current supported in $W$, then
\begin{equation}
\mathbf M_W(\partial \mathbf E)
\le \mathbf M_W(\partial \mathbf E+\partial \mathbf X) + C\,\mathbf M_W(\mathbf X).
\end{equation}
In the context of the Jang surface, the constant $C$ depends only on the local mean curvature bound coming
from the regularized Jang equation.

\medskip

Fix now a relatively compact coordinate ball
\begin{equation}
U \subset M^n\times \mathbb R,
\end{equation}
and identify it with a bounded open subset of $\mathbb R^{n+1}$ by a smooth chart.
Since Minkowski dimension is invariant under bi-Lipschitz changes of coordinates,
it is enough to work in this Euclidean chart.

\medskip

Let $E \subset U$ be a set of finite perimeter whose boundary has support $\bar{\Sigma}\cap U$.
We claim that $E$ satisfies the almost minimizing hypothesis of
\cite[(6.8)--(6.9)]{FocardiMarcheseSpadaro}. To see this, suppose $F$ is another set of finite perimter with
\begin{equation}
E\Delta F \Subset B_r(x)\subset U.
\end{equation}
Choose the filling current $\mathbf X:=\llbracket F\rrbracket-\llbracket E\rrbracket$.
Then $\partial \mathbf X=\partial \llbracket F\rrbracket-\partial \llbracket E\rrbracket$ in $B_r(x)$, and
\begin{equation}
\mathbf M(\mathbf X)=|E\Delta F|.
\end{equation}
Applying the $C$-almost minimizing property gives
\begin{equation}
\Per(E;B_r(x))
\le \Per(F;B_r(x)) + C\,|E\Delta F|.
\end{equation}
Since $E\Delta F\Subset B_r(x)\subset \mathbb R^{n+1}$ we have
\begin{equation}
|E\Delta F| \le \beta_{n+1} r^{n+1},
\end{equation}
and therefore
\begin{equation}
\Per(E;B_r(x))
\le \Per(F;B_r(x)) + C\beta_{n+1} r^{n+1},
\end{equation}
where $\beta_{n+1}$ is the Euclidean volume of the unit $(n+1)$-ball.
Thus $E$ is a perimeter almost minimizer in the sense of
\cite{FocardiMarcheseSpadaro}, with
\begin{equation}
\alpha(r):=C\beta_{n+1}r,
\end{equation}
so that
\begin{equation}
\Per(E;B_r(x)) \le \Per(F;B_r(x)) + \alpha(r)\,r^n.
\end{equation}
Moreover $\alpha(r)$ is nondecreasing, $\alpha(r)\to 0$ as $r\downarrow 0$, and
$\alpha(r)/r$ is constant (hence nonincreasing), together with
\begin{equation}
\int_0^T \frac{\alpha(t)^{1/2}}{t}\,dt
= \sqrt{C\beta_{n+1}}\int_0^T t^{-1/2}\,dt
<\infty.
\end{equation}
It follows that all assumptions of \cite[Theorem 6.7]{FocardiMarcheseSpadaro} are satisfied.

\medskip

Applying \cite[Theorem 6.7]{FocardiMarcheseSpadaro} in the chart, with ambient
dimension $n+1$, yields the Minkowski dimension estimate
\begin{equation}
\dim_{\mathcal M} (\Sing(\bar{\Sigma}\cap U) )\le (n+1)-8 = n-7.
\end{equation}
Since $U\subset M^n\times \mathbb R$ is arbitrary, the same estimate holds locally
on all of $\bar{\Sigma}$, and this proves
\begin{equation}
\dim_{\mathcal M} ( \Sing(\bar{\Sigma}))\le n-7.
\end{equation}
\end{proof}

\section{Desingularized Jang Graph and Proof of Main Theorems} \label{sec:singularities}

\begin{lemma}\label{lem:jang-stability}
Let $(M^n,g,k)$, $n\geq 4$ be either asymptotically flat with harmonic asymptotics, or asymptotically hyperboloidal with Wang asymptotics, 
satisfying a strict dominant energy condition with $\mu-|J|_g\geq \lambda r^{-n-1}$ in the asymptotic end for some constant $\lambda>0$. Consider a regular Jang graphical component
\begin{equation*}
\Sigma_\infty=\graph(f_\infty,U_\infty)\subset \Reg(\bar\Sigma)
\end{equation*}
obtained from Theorem \ref{thm:higher-dim-jang-limit} or
\ref{thm:higher-dim-jang-limit-lundberg}, and let $\bar{g}$ be its induced metric.
Then there exist smooth positive functions $\rho$ and $Q$ on $\Sigma_{\infty}$ and $c, C_l\in\mathbb{R}$ such that
\begin{equation}\label{Qrho}
|\nabla^l \left(\rho -(1+cr^{2-n})\right)|_{\delta}\leq C_l r^{1-n-l},\quad |\nabla^l Q|_{\delta}\leq C_l r^{-n-1 -l},
\quad Q\geq \frac{\lambda}{2}r^{-n-1},
\end{equation}
for large $r$. Moreover, the following inequality holds for all $\phi\in C^{\infty}(\Sigma_{\infty})$ with the property that $\phi$ vanishes
in a neighborhood of the singular set, is constant in the asymptotically flat end, and has bounded supported outside this same end: 
\begin{align}\label{bw}
\begin{split}
\lim_{r\rightarrow\infty}&\int_{\Sigma^r_\infty}\left(\rho|\nabla \phi|_{\bar{g}}^2
+\frac12  \rho
\left(
R_{\bar{g}}
-2\Delta_{\bar{g}}\log \rho
-\frac{n+1}{n+2} |\nabla \log \rho|_{\bar{g}}^2
\right) \phi^2 \right)dV_{\bar{g}}\\
&\geq
\int_{\Sigma_\infty} \rho \,Q \,\phi^2 dV_{\bar{g}},
\end{split}
\end{align}
where $\Sigma_{\infty}^r$ denotes the region contained within coordinate sphere $S_r$.
\end{lemma}

\begin{remark}
A limit appears on the left-hand side of \eqref{bw} due to the possible lack of integrability
of the Jang scalar curvature in the asymptotically hyperboloidal setting.
\end{remark}

\begin{proof}
We will treat the asymptotically hyperboloidal case first, and will assume for simplicity that the initial 
data possesses a single end. Let $\rho$ be a smooth positive function to be chosen. Utilizing the
Jang scalar curvature formula, Proposition \ref{SY scalar curvature id}, while integrating $\mathrm{div}_{\bar{g}}X$ and $\Delta_{\bar{g}}\log\rho$
by parts produces
\begin{align}\label{2}
&\int_{\Sigma_{\infty}^r}\left(\rho|\nabla \phi|_{\bar{g}}^2 +\frac{\rho}{2}\left(R_{\bar{g}}-2\Delta_{\bar{g}}\log\rho -\frac{n+1}{n+2}|\nabla\log\rho|_{\bar{g}}^2\right)\phi^2\right)dV_{\bar{g}} \nonumber\\
\geq&\int_{\Sigma_\infty^r}\left(\rho|\nabla \phi|_{\bar{g}}^2 +(\mu-|J|_g)\rho \phi^2+\rho \phi^2 |X|_{\bar{g}}^2 +\left(1-\frac{n+1}{2(n+2)}\right)\phi^2 \frac{|\nabla \rho|_{\bar{g}}^2}{\rho}\right)dV_{\bar{g}}\nonumber\\
&\qquad+\int_{\Sigma^r_\infty}\left(2\phi\langle\nabla \phi,\nabla\rho\rangle +\phi^2 \langle\nabla\rho,X\rangle +2\rho \phi\langle\nabla \phi,X\rangle\right)dV_{\bar{g}}\\
&\qquad\qquad-\int_{S_r}\left(\rho \phi^2 \langle X,\nu\rangle +\phi^2 \nu(\rho)\right)dA,\nonumber
\end{align}
where $S_r$ denotes a coordinate sphere with unit outer normal $\nu$ in the asymptotically flat end of $\Sigma_\infty$. Note that the Jang scalar curvature is not necessarily integrable in the asymptotically flat end, due to the expansion \cite[Lemma B.3]{LundbergHyperbolic}
\begin{equation}
R_{\bar{g}}=2(n-2)\frac{\Delta_{S^{n-1}}\boldsymbol{\alpha}}{r^{n}}+O(r^{-n-1+\epsilon})    
\end{equation}
for any $\epsilon>0$ and some $\boldsymbol{\alpha} \in C^{\infty}(S^{n-1})$.  However, the integral on the left-hand side of \eqref{2} is finite even in the limit as $r\rightarrow\infty$ since the leading term of this expansion integrates to zero on coordinate spheres. It follows that
\begin{align}\label{2.1}
\begin{split}
&\int_{\Sigma_\infty^r}\left(\rho|\nabla \phi|_{\bar{g}}^2 +\frac{\rho}{2}\left(R_{\bar{g}}-2\Delta_{\bar{g}}\log\rho -\frac{n+1}{n+2}|\nabla\log\rho|_{\bar{g}}^2\right)\phi^2\right)dV_{\bar{g}}\\
\geq&\int_{\Sigma_\infty^r}\left((\mu-|J|_g)\rho \phi^2 +\rho|\nabla \phi +\phi (X +\nabla\log\rho)|_{\bar{g}}^2\right)dV_{\bar{g}}\\
&-\int_{\Sigma_\infty^r}\left(\left(1-\frac{n+3}{2(n+2)}\right)\phi^2 \frac{|\nabla\rho|_{\bar{g}}^2}{\rho} +\phi^2 \langle\nabla\rho, X\rangle\right)dV_{\bar{g}}\\
&
-\int_{S_r}\rho \phi^2 \left(\langle X,\nu\rangle +\nu(\log\rho)\right)dA.
\end{split}
\end{align}
According to \cite[(B.15)]{LundbergHyperbolic} the flux density is 
\begin{equation}
\langle X,\nu\rangle=(n-2)(n-3)\frac{\boldsymbol{\alpha}}{r^{n-1}}+O(r^{-n +\epsilon}).
\end{equation}
This motivates the choice
\begin{equation} \label{rho alpha}
\rho=1+\psi \cdot \frac{(n-3)\alpha_0}{r^{n-2}},\quad\quad\quad \alpha_0 =\frac{1}{\omega_{n-1}}\int_{S^{n-1}}\boldsymbol{\alpha},
\end{equation}
where $\psi$ is a smooth nonnegative cut-off function which vanishes inside $S_{r_0}$ and is 1 outside $S_{2r_0}$ in the asymptotic end, and satisfies $|\nabla\psi|_{\bar{g}}\leq 2r_0^{-1}$. By \cite[(B.15) and (B.18)]{LundbergHyperbolic} we have $|X|_{\bar{g}}=O(r^{1-n})$, and thus the strict DEC
assumption $\mu-|J|_g\geq \lambda r^{-n-1}$ implies that
\begin{equation}
\mu -|J|_g\geq \left(1-\frac{n+3}{2(n+2)}\right)|\nabla\log\rho|^2 +\langle\nabla\log\rho, X\rangle
\end{equation}
if $r_0$ is chosen sufficiently large. Moreover, with this choice of $\rho$ we find that for $r$ large enough 
\begin{equation}
\langle X,\nu\rangle +\nu(\log\rho)=(n-2)(n-3)\frac{\boldsymbol{\alpha}-\alpha_0}{r^{n-1}}+O(r^{-n +\epsilon}).
\end{equation}
Hence, the flux integral in \eqref{2.1} converges to zero, and we obtain
\begin{align}\label{2.2}
\begin{split}
\lim_{r\rightarrow\infty}&\int_{\Sigma_\infty^r}\left(\rho\,|\nabla \phi|_{\bar{g}}^2 +\frac{\rho}{2}\left(R_{\bar{g}}-2\Delta_{\bar{g}}\log\rho -\frac{n+1}{n+2}|\nabla\log\rho|_{\bar{g}}^2\right)\phi^2\right)dV_{\bar{g}}\\
\geq&\int_{\Sigma_\infty}\left( \rho\, Q\, \phi^2  +\rho|\nabla \phi +\phi (X +\nabla\log\rho)|_{\bar{g}}^2\right)dV_{\bar{g}},
\end{split}
\end{align}
for some positive smooth function $Q$ satisfying \eqref{Qrho} such that $Q\geq \frac{\lambda}{2} r^{-n-1}$ for large $r$.

\medskip

In the case of asymptotically flat initial data, an analogous argument holds with many simplifications. In particular,
we may choose $\rho=1$, the Jang scalar curvature is integrable in the asymptotically flat end, and there are no flux terms arising from the vector field $X$.
\end{proof}

Lemma~\ref{lem:jang-stability} is the analogue, in the present Jang--graph
setting, of \cite[Corollary~3.31]{BrendleWang}. It provides the coercive
estimate needed for the conformal blow-up argument. 

\medskip

We follow \cite[Section 3.7]{BrendleWang} to blow up the singular set in order to obtain a complete manifold without singularities. However, in our setting the singular set may not be compact, which differs from the situation in \cite{BrendleWang}. Therefore, we shall modify the proof accordingly.

\begin{theorem}\label{thm: singularities cylinder} 
Let $(M^n,g,k)$ be as in Lemma \ref{lem:jang-stability}, with
$(\Sigma_\infty,\bar{g})$ denoting the associated Jang graph. In what follows, $\nabla$ and $\Delta$ will be the gradient and Laplacian on $(\Sigma_\infty,\bar{g})$.
    \begin{enumerate}
        \item[\rm (i)] Let $\check{M}=M^n\times \mathbb{R}$ be equipped with the product metric $\check{g}$, $\overline{\Sigma}_\infty$ is the closure of $\Sigma_\infty$ in $\check{M}$, and let $\mathcal{S}$ denote the singular set in $\overline{\Sigma}_\infty$. Then $\mathcal{S}$ can be decomposed into a union of compact sets $\mathcal{S}_i$ such that \[d_{(\check{M},\check{g})}(\mathcal{S}_i,\mathcal{S}_j)\ge |i-j|-1.\]
        \item[\rm (ii)]  There exists a function $\Psi_i\in C^2(\check{M}\setminus \mathcal{S}_i)$ supported on a small neighborhood of $\mathcal{S}_i$ such that for $x\in \Sigma_\infty$ the following statements hold.
        \begin{enumerate}
            \item[\rm (a)] There exists a constant $C_i>1$ for which
        \[\Delta \Psi_i+\frac{n-2}{n+2}\langle \nabla \log \rho,\nabla \Psi_i\rangle\le C_i.\]  
            \item[\rm (b)] If $d_{(\check{M},\check{g})}(x,\mathcal{S}_i) $ is sufficiently small, then
        \begin{equation*}
        \begin{split}
            &\Delta \Psi_i+\frac{n-2}{n+2}\langle \nabla \log \rho,\nabla \Psi_i\rangle<0, \\ 
            &\Psi_i(x)\ge \frac{1}{2(n-2)}3^{2-n}d_{(\check{M},\check{g})}(x,\mathcal{S}_i)^{-2}. 
        \end{split} 
        \end{equation*}
        \end{enumerate}
        \item[\rm (iii)] Let $\overline{\Psi}=\sum_{i=1}^\infty C_i^{-1}\Psi_i$.  Then there exists $\varepsilon_0>0$ such that on $\Sigma_\infty$ the following inequality holds
        \begin{equation*}
            \varepsilon_0 \left(\Delta \overline\Psi+\frac{n-2}{n+2}\langle \nabla \log \rho,\nabla \overline\Psi\rangle\right) \le \frac{n-2}{4(n+2)}Q,
            \end{equation*}
            where $Q$ is given by Lemma \ref{lem:jang-stability}.
                 \item[\rm (iv)] Let $w=1+\varepsilon_0\overline{\Psi}$ and $\tilde{g}=w^\frac{n+2}{n-2}\bar{g}$, then $(\Sigma_\infty, \tilde{g})$ is complete.  
    \end{enumerate}
\end{theorem}

\begin{proof}
    (i) Note that the singular set can only concentrate along the cylindrical ends, and the asymptotically flat end is regular 
    according to Theorems \ref{thm:higher-dim-jang-limit} and \ref{thm:higher-dim-jang-limit-lundberg}. Set 
    \begin{equation}
        \Omega_i:=\overline{\{ t_{i-1}\le |f_\infty|\le t_i\}},\qquad \mathcal{S}_i:=\mathcal{S}\cap \Omega_i,
    \end{equation} 
    with $t_0 = 0$ and $t_i$ ($i>0$) chosen large enough so that $d_{(\check{M},\check{g})}(\Omega_i, \Omega_j) \ge |i-j|-1$. Therefore, each $\mathcal{S}_i$ is compact and $d_{(\check{M},\check{g})}(\mathcal{S}_i, \mathcal{S}_j)\ge |i-j|-1$.

\medskip

    (ii) Following \cite[page 31]{BrendleWang}, choose $t_* \in (0,1)$ such that $\sqrt{t_*} \le \frac{1}{2}\operatorname{inj}_{(\check{M},\check{g})}(p)$ for all $p \in \mathcal{S}$. Construct $\Psi_i$ as in \cite[page 32]{BrendleWang}. Properties (a) and (b) then follow from \cite[Proposition 3.34, 3.35, 3.36]{BrendleWang}.

    \medskip
    
    (iii) Since $\Psi_i$ is supported on a small neighborhood of $\mathcal{S}_i$, it follows that $\Psi_i\Psi_{j}=0$ whenever $|i-j|>1$. 
     Note that by Lemma \ref{lem:jang-stability}, 
     $Q$ is uniformly positive on cylindrical ends, and thus there exists a suitable $\varepsilon_0>0$ as required.

\medskip
     
    (iv) Let $y_0, y_1\in \Sigma_\infty$ and let $\sigma:[0,1]\to \Sigma_\infty$ be a smooth path connecting them. Define
    \begin{equation}
     I:=\{i : \Omega_i\cap \sigma([0,1])\neq \emptyset\}.
    \end{equation}
   Let $c_i$ and $\varepsilon_i$ be positive constants such that
    \begin{equation}
    w(x)\ge \varepsilon_i d_{(\check{M},\check{g})}(x,\mathcal{S})^{-2}\quad \text{for }  d_{(\check{M},\check{g})}(x,\mathcal{S}_i)\in (0, c_i].
    \end{equation}
    Thus, on the region $\cup_{i\in I}\Omega_i$, we have 
    \begin{equation}
    w(x)\ge \min_{i\in I}\{\varepsilon_i\} d_{(\check{M},\check{g})}(x,\mathcal{S})^{-2}\quad \text{for }  d_{(\check{M},\check{g})}(x,\mathcal{S}_i)\in (0, \min_{i\in I}\{c_i\}].
    \end{equation}
Following \cite[Proposition 3.41]{BrendleWang}, this implies
   \begin{equation}
   \begin{split}
       d_{(\check{M},\check{g})}(y_1,\mathcal{S}_l)^{-\frac{4}{n-2}}\le& \max\left\{d_{(\check{M},\check{g})}(y_0,\mathcal{S}_l)^{-\frac{4}{n-2}},\max_{i\in I}\{c_i^{-\frac{4}{n-2}}
       \}\right\}
       \\ &+\frac{4}{n-2}\max_{i\in I}\left\{\varepsilon_i^{-\frac{n+2}{2(n-2)}}\right\}\int_0^1|\sigma'(t)|_{\tilde{g}}dt,
   \end{split}
   \end{equation}
where $l$ satisfies $y_1\in\Omega_l$.
      On the other hand, since $d_{(\check{M},\check{g})}(\Omega_i, \Omega_j)\ge |i-j|-1$, we have $|\sigma|_{\tilde{g}}\ge |I|-1$. 
   Therefore, $d_{(\Sigma_\infty,\tilde{g})}(y_0,y_1)\to \infty$ when $d_{(\check{M},\check{g})}(y_1,\mathcal{S}_l)\to 0$. Hence, $\tilde{g}$ is complete.
\end{proof}

\begin{proof}[Proof of Theorem \ref{thm:main2}] 
The result is known in dimensions $3\leq n\leq 7$, so assume that $n\geq 8$. By Corollary \ref{cor:AH-reduction}, to prove $E\ge |P|$ it suffices to show that $E\ge 0$. By Theorem \ref{DS density} it further suffices to restrict attention to Wang asymptotics, and assume a strict dominant energy condition. Wang asymptotics allow for an application of Theorem \ref{thm:higher-dim-jang-limit-lundberg} to construct a Jang graph $(\Sigma_{\infty},\bar{g})$ with an asymptotically flat end having energy
$\bar{E}_{ADM}=(n-1)E$. According to Lemma \ref{lem:jang-stability} the Jang graph satisfies \eqref{bw}.
Moreover, this manifold may be made complete by applying the conformal transformation of Theorem \ref{thm: singularities cylinder} to obtain a complete Jang graph $(\Sigma_{\infty},\tilde{g})$, whose asymptotically flat end is preserved without change. Observe that Theorem \ref{thm: singularities cylinder} (iii) implies that the inequality \eqref{bw} still holds for this conformal Jang graph by \cite[Proposition 3.40]{BrendleWang}, with modified choices $\tilde{\rho}$, $\tilde{Q}$ of $\rho$, $Q$. In fact, the slightly stronger inequality \eqref{2.2'} holds and $\tilde{Q}\geq \frac{\lambda}{4}r^{-n-1}$ in the asymptotically flat end. We can then utilize Proposition \ref{Rg} to find a perturbation $(\Sigma_{\infty},\hat{g})$ of the conformal Jang graph with Schwarzschild asymptotics, while preserving the ADM energy. Additionally, Proposition \ref{Rg} guarantees that $(\Sigma_{\infty},\hat{g})$ satisfies \eqref{bw}. This manifold is then a so called $n$-dataset as defined by Brendle-Wang, and therefore \cite[Theorem 1.5]{BrendleWang} implies that
the following mass quantity is nonnegative
\begin{equation}
(n-1)\hat{c}+2(n-3)\alpha_0\geq 0,
\end{equation}
where $\alpha_0$ arises from \eqref{rho alpha}, and $\hat{g}=(1+\hat{c}r^{2-n})\delta$ in the asymptotic end. Note that $\hat{E}_{ADM}=\frac{n-2}{2}\hat{c}$, and by \cite[(3.4) and (C.1)]{LundbergHyperbolic} we find $\alpha_0 =-\frac{1}{n-3}\hat{E}_{ADM}$.
It follows that
\begin{equation}
0\leq (n-1)\hat{c}+2(n-3)\alpha_0 =\frac{2}{n-2}\hat{E}_{ADM}.
\end{equation}
Since $E=\frac{\bar{E}_{ADM}}{n-1}$ and $\bar{E}_{ADM}=\hat{E}_{ADM}$, we conclude that $E\geq 0$. Finally, the case of equality statement follows from \cite{HirschJangZhangHyperboloidal,HuangJangMartin}.
\end{proof}

\begin{proof}[Proof of Theorem \ref{thm:main}]
The result is known in dimensions $3\leq n\leq 7$, so assume that $n\geq 8$.
To prove the inequality portion of the result, we proceed by contradiction and assume that
$E_{\mathrm{ADM}}< |P_{\mathrm{ADM}}|$. By Theorem \ref{thm:reduction} the initial data may then be
perturbed (keeping the same notation) to achieve a strict dominant energy condition, harmonic asymptotics, and $E_{ADM}<0$.
According to Theorem \ref{thm:higher-dim-jang-limit} and Lemma \ref{lem:jang-stability}, there exists an associated Jang graph $(\Sigma_{\infty},\bar{g})$ which satisfies \eqref{bw} and has an asymptotically flat end of the same ADM energy. 
This manifold may not be complete. However, we may then apply the conformal transformation of Theorem \ref{thm: singularities cylinder} to obtain a complete Jang graph $(\Sigma_{\infty},\tilde{g})$, whose asymptotically flat end is preserved without change. In light of Theorem \ref{thm: singularities cylinder} (iii), the inequality \eqref{bw} still holds for this conformal Jang graph by \cite[Proposition 3.40]{BrendleWang}, with modified choices $\tilde{\rho}$, $\tilde{Q}$ of $\rho$, $Q$. In fact, the slightly stronger inequality \eqref{2.2'} holds and $\tilde{Q}\geq \frac{\lambda}{4}r^{-n-1}$ in the asymptotically flat end.
Next, observe that Proposition \ref{Rg} provides a perturbation $(\Sigma_{\infty},\hat{g})$ of the conformal Jang graph with Schwarzschild asymptotics, while maintaining negative ADM energy $\hat{E}_{ADM}<0$. Moreover, Proposition \ref{Rg} guarantees that $(\Sigma_{\infty},\hat{g})$ satisfies \eqref{bw}. This manifold is then a so called $n$-dataset as defined by Brendle-Wang, and therefore \cite[Theorem 1.5]{BrendleWang} implies that $\hat{E}_{ADM}\geq 0$, yielding a contradiction.

\medskip

Now consider the case $E_{\mathrm{ADM}}=|P_{\mathrm{ADM}}|$.
We apply \cite{HirschHuang} to deduce that $(M^n,g,k)$ embeds into a spacetime $(\mathbf M^{n+1},\mathbf g)$ containing a null parallel vector field in case $(g, k)\in C^5_{\mathrm{loc}}(M^n)\times C^4_{\mathrm{loc}}(M^n)$.
In this case, we may apply \cite{HirschZhang2} to find that the ambient metric may be written in Brinkmann
coordinates as
\begin{equation}
\mathbf{g}
=
-2\,dt\,du
+
F(x,u)\,du^2
+
(dx^1)^2+\cdots+(dx^{n-1})^2,
\end{equation}
where $F$ is independent of $t$ and satisfies
\begin{equation}
\Delta_{\mathbb R^{n-1}}F(\cdot,u)\le 0
\qquad\text{for every }u\in\mathbb R.
\end{equation}
This proves the claim.
\end{proof}

\section{Asymptotically Hyperboloidal Positive Mass Theorem Via the Asymptotically Flat Positive Mass Theorem}\label{sec:hyperbolic}

In this section we explain that Theorem \ref{thm:main} implies
Theorem \ref{thm:main2} under stronger asymptotic assumptions at infinity.

\subsection{Proof for Wang asymptotics}

In the Riemannian $(k=g)$ case, the desired inequality for asymptotically hyperboloidal
manifolds with Wang asymptotics follows directly from the recent work of
Chru\'sciel and Delay \cite{ChruscielDelayHyperbolic}. More precisely, they show
that the causal-future-directed character of the asymptotically hyperboloidal
energy--momentum vector can be reduced to the spacetime positive mass theorem
for asymptotically Euclidean initial data sets.

\medskip

Therefore, combining their reduction with Theorem \ref{thm:main}, we obtain the
hyperbolic positive mass theorem for asymptotically hyperboloidal Riemannian
manifolds with Wang asymptotics. In particular, if $(M^n,g)$ is an
asymptotically hyperbolic manifold with Wang asymptotics and scalar curvature
\begin{equation}
R_g\ge -n(n-1),
\end{equation}
then its energy--momentum vector $(E,P)$ satisfies
\begin{equation}
E\ge |P|.
\end{equation}

\subsection{Proof for exact AdS--Schwarzschild asymptotics}

We now explain how Theorem \ref{thm:main} yields the hyperbolic positive mass
theorem for initial data sets which are exactly AdS--Schwarzschild near
infinity.

\medskip

Assume that $(M^n,g,k)$ is an asymptotically hyperboloidal initial data set
satisfying the dominant energy condition, and that for some compact set
$\mathcal{C}\subset M^n$ the exterior region $M^n\setminus \mathcal{C}$ is realized as an umbilic spacelike
hypersurface in the Schwarzschild spacetime of mass $m$, with induced metric
equal to that of the corresponding AdS--Schwarzschild constant time slice.

\medskip

Since both the asymptotically hyperboloidal AdS--Schwarzschild slice and the
standard asymptotically flat Schwarzschild slice occur in the same ambient
Schwarzschild spacetime, we may bend the given hypersurface in the exterior
region so as to replace the hyperboloidal end by an asymptotically flat end,
while keeping the data unchanged on a sufficiently large compact set. In this
way one obtains an asymptotically flat initial data set
\begin{equation}
(\widetilde M^n,\widetilde g,\widetilde k),
\end{equation}
which still satisfies the dominant energy condition and whose asymptotically
flat end is exactly Schwarzschild of mass $m$.
Applying Theorem \ref{thm:main} to
$(\widetilde M^n,\widetilde g,\widetilde k)$ we conclude that $m\ge0$, and in particular $E\ge|P|$.

\appendix

\section{Density Lemma}\label{appendix}

The purpose of this section is to perform a deformation applicable to the Jang graph metric, which produces a Schwarzschild end while preserving the weighted inequality for the scalar curvature \eqref{bw} (cf. Proposition \ref{Rg}), and ensuring that the total energy remains unchanged.

\medskip

Let $(\Sigma_{\infty},\bar{g})$ be the Jang graph as in the context of Lemma \ref{lem:jang-stability}.
Although this is not necessarily complete, we may apply the conformal transformation of Theorem \ref{thm: singularities cylinder} to obtain a complete Jang graph $(\Sigma_{\infty},\tilde{g})$, whose asymptotically flat end is preserved without change. Moreover, Theorem \ref{thm: singularities cylinder} (iii) implies that inequality \eqref{2.2} still holds for this conformal Jang graph by the proof of \cite[Proposition 3.40]{BrendleWang}, with modified functions $\tilde{\rho}=w^{-\frac{n+2}{2}}\rho$ and $\tilde{Q}=\frac{1}{2}w^{-\frac{n+2}{n-2}}Q$. In particular
\begin{equation}\label{2.2'}
\lim_{r\rightarrow\infty}I_r(\tilde{g})[\phi] 
\geq\int_{\Sigma_\infty}\left( \tilde{\rho}\, \tilde{Q}\, \phi^2  +\tilde{\rho}|\nabla \phi +\phi (X +\nabla\log\tilde{\rho})|_{\tilde{g}}^2\right)dV_{\tilde{g}},
\end{equation}
where
\begin{equation}
I_r(\tilde{g})[\phi]\!:=\!\!\int_{\Sigma^r_\infty}\!\!\left(\tilde{\rho}|\nabla \phi|_{\tilde{g}}^2 +\frac{\tilde{\rho}}{2}\left(R_{\tilde{g}}-2\Delta_{\tilde{g}}\log\tilde{\rho} -\frac{n+1}{n+2}|\nabla\log\tilde{\rho}|_{\tilde{g}}^2\right)\phi^2 \right)dV_{\tilde{g}}.
\end{equation}

\begin{proposition} \label{Rg}
Let $(\Sigma_{\infty},\tilde{g})$ be the complete conformal Jang graph as described above. Then
for each sufficiently large \(r_0\in\mathbb{R}_+\), there exists a perturbation $\hat{g}$ of $\tilde{g}$ with the
following properties:
\[
        \hat g=\tilde g
        \qquad\text{on } \{r\leq r_0\},
\]
\[
        \hat g
        =
        (1+\hat{c} r^{2-n})\delta
        \qquad\text{on } \{r\geq 2r_0\},
\]
with $\hat{c}\in \mathbb{R}$ chosen so that
\[
        E_{\mathrm{ADM}}(\hat g)=E_{\mathrm{ADM}}(\tilde g).
\]
Furthermore
\begin{equation*}
       I(\hat{g})[\phi]:=\lim_{r\rightarrow\infty}I_r(\hat{g})[\phi]\ge \int_{\Sigma_\infty}\hat{\rho}\hat{Q}\phi^2 dV_{\hat{g}},
\end{equation*}
for all test functions $\phi$ as in Lemma \ref{lem:jang-stability} where $\hat{\rho}=\tilde{\rho}$ and $\hat{Q}=\frac{1}{2}\tilde{Q}$.
\end{proposition}

\begin{remark}
Note that since $R_{\hat{g}}\in L^1$ in the asymptotically flat end, it follows that $I(\hat{g})$ is well-defined and invariant under any choice of exhaustion.
\end{remark}

\begin{proof}
Let \(\chi\in C^\infty([0,\infty))\) satisfy
\begin{equation}
        \chi(s)=0 \quad\text{for } s\leq 1,
        \qquad
        \chi(s)=1 \quad\text{for } s\geq 2,
        \qquad
        0\leq \chi\leq 1,
\end{equation}
and set $\chi_{r_0}(r):=\chi(r/r_0)$ so that
\begin{equation*}
        \chi_{r_0}=0 \quad\text{for } r\leq r_0,
        \quad
        \chi_{r_0}=1 \quad\text{for } r\geq 2r_0,\quad
        |\partial^l \chi_{r_0}|
        \leq C_l r^{-l} \quad\text{for } r_0\leq r\leq 2r_0.
\end{equation*}
Define \(\tilde g\) on the asymptotically flat end by
\begin{equation}
        \hat g
        :=
        (1-\chi_{r_0})\tilde g+\chi_{r_0}(1+\hat{c}r^{2-n})\delta,
\end{equation}
and set \(\hat g=\tilde g\) away from this end. By choosing $\hat{c}=\frac{2}{n-2}E_{ADM}(\bar{g})$
we find that $E_{ADM}(\hat{g})=E_{ADM}(\tilde{g})$. It remains to establish the desired lower bound
for $I(\hat{g})$.

\medskip

We begin by comparing functionals with respect to the two metrics. Set
\begin{equation}
\mathcal{R}_{\hat{g}}:=R_{\hat g}-2\Delta_{\hat g}\log\hat{\rho}
-c_n|\nabla\log\hat{\rho}|_{\hat{g}}^2 ,\qquad c_n =\frac{n+1}{n+2},\qquad \hat{\rho}=\tilde{\rho},
\end{equation}
and observe that
\begin{align}\label{Ig-Ig}
I_r(\hat g)[\phi]-I_r(\tilde g)[\phi]
&=\!
\int_{\Sigma^r_\infty}
\left(
    \tilde{\rho} |\nabla\phi|_{\hat g}^2
    +\frac{\tilde{\rho}}{2}\mathcal R_{\hat g}\phi^2
\right)
\left(
    \sqrt{\det \tilde{g}^{-1}\hat g}-1
\right)dV_{\tilde{g}}
\notag\\
&+
\int_{\Sigma^r_\infty}
\left[
    \tilde{\rho}(\hat g^{ij}-\tilde g^{ij})\phi_i\phi_j
    +\frac{\tilde{\rho}}{2}(R_{\hat g}-R_{\tilde g})\phi^2
\right]dV_{\tilde g}
\\
&-\!\!
\int_{\Sigma^r_\infty}\!\!\!
\frac{\tilde{\rho}\phi^2}{2}
\!\left[
    2(\Delta_{\hat g}-\Delta_{\tilde g})\log\tilde{\rho}
    +c_n(\hat g^{ij}\!-\tilde g^{ij})
        (\log\tilde{\rho})_i(\log\tilde{\rho})_j
\right]\!dV_{\tilde g}.\notag
\end{align}

\medskip

We will first estimate the scalar curvature integral. Write $\hat h:=\hat g-\delta$ and $\tilde h:=\tilde g-\delta$, 
and recall that the expansion of scalar curvature about the Euclidean metric is given by
\begin{equation}
        R_{\delta+h}
        =
        L_\delta h+\mathcal{Q}(h,\partial h,\partial^2h),
\end{equation}
where the linearization and quadratic error term satisfy 
\begin{equation}
        L_\delta h
        =
        \partial_i\partial_j h^{ij}
        -
        \Delta_\delta(\operatorname{tr}_\delta h),
\qquad
        |\mathcal{Q}(h,\partial h,\partial^2h)|
        \leq
        C\bigl(|h||\partial^2h|+|\partial h|^2\bigr).
\end{equation}
Since $\hat h, \tilde h=O_2(r^{2-n})$ with constants independent of large \(r_0\), we have
\begin{equation}
       \mathcal{Q}(\hat h,\partial\hat h,\partial^2\hat h)
        =
        O(r^{2-2n}),
\qquad
      \mathcal{Q}(\tilde h,\partial\tilde h,\partial^2\tilde h)
        =
        O(r^{2-2n}).
\end{equation}
Utilizing $\hat{g}=\tilde{g}$ for $r\le r_0$ and
\begin{equation}
R_{\hat{g}}-R_{\tilde{g}}=L_{\delta}(\hat{h}-\tilde{h})+O(r^{2-2n}),
\end{equation}
together with an integration by parts of $L_\delta$ produces
\begin{align*}
\begin{split}
    &\int_{\Sigma^r_\infty} \tilde{\rho}\phi^2(R_{\hat{g}}-R_{\tilde{g}})dV_{\tilde{g}}
    \\=& -\int_{\Sigma^r_\infty\setminus \Sigma^{r_0}_\infty} \left[\tilde{\rho}\phi\partial_i\phi\left(\partial_j(\hat{h}^{ij}-\tilde{h}^{ij})-\partial^i \operatorname{tr}_{\delta}(\hat{h}-\tilde{h})\right)+\tilde{\rho}\phi^2 O(r^{2-2n})\right]dV_{\tilde{g}}\\
    &+o(1).
\end{split}
\end{align*}
Here, $o(1)$ represents the boundary integral at $\partial \Sigma_{\infty}^r$ which tends to zero as $r\rightarrow\infty$, since $E_{ADM}(\hat{g})=E_{ADM}(\tilde{g})$.

\medskip

Applying the decay of all relevant terms to \eqref{Ig-Ig} yields
\begin{align*}
\begin{split}
I_r(\hat{g})[\phi]\geq& \, I_r(\tilde{g})[\phi] +o(1)\\
&+\int_{\Sigma^r_\infty\setminus \Sigma^{r_0}_\infty}\left(O(r^{2-n})\tilde{\rho} |\nabla \phi|_{\tilde{g}}^2 +O(r^{1-n})\tilde{\rho} \phi|\nabla \phi|_{\tilde{g}}+O(r^{2-2n})\tilde{\rho} \phi^2\right)dV_{\tilde{g}}.
\end{split}
\end{align*}
Now employ \eqref{2.2'} together with the Young's inequality
\begin{equation}
r^{1-n}\tilde{\rho} \phi|\nabla \phi|_{\tilde{g}}\leq \frac{r^{3.1-n}}{2}\tilde{\rho}|\nabla \phi|_{\tilde{g}}^2 +\frac{r^{-n-1.1}}{2}\tilde{\rho} \phi^2,
\end{equation}
and send $r\to \infty$ to find
\begin{equation}
\begin{split}
I(\hat{g})[\phi]\geq& \int_{\Sigma_\infty}\left(\frac{4}{5}\tilde{\rho}\tilde{Q} \phi^2 +\tilde{\rho}|\nabla \phi +\phi(X +\nabla\log\tilde{\rho})|_{\tilde{g}}^2 \right)dV_{\tilde{g}}
\\&- \int_{\Sigma_\infty\setminus\Sigma^{r_0}_\infty}C  r^{3.1-n}\tilde{\rho} |\nabla \phi|_{\tilde{g}}^2 dV_{\tilde{g}}
\end{split}
\end{equation}
for some $C >0$, where we have used $\tilde{Q}\ge \frac{\lambda}{4}r^{-n-1}$ and have chosen $r_0$ large enough to absorb the error terms involving $\tilde{\rho}\phi^2$. Since
\begin{equation}
\begin{split}
&|\nabla \phi|^2_{\tilde{g}}
\\=\,&|\nabla \phi +\phi(X +\nabla\log\tilde{\rho})|^2_{\tilde{g}} -2\phi\nabla \phi\cdot (X+\nabla\log\tilde{\rho})-\phi^2|X+\nabla\log\tilde{\rho}|^2_{\tilde{g}}\\
\leq \, & |\nabla \phi +\phi(X +\nabla\log\tilde{\rho})|_{\tilde{g}}^2 +\frac{1}{2}|\nabla \phi|^2_{\tilde{g}}+\phi^2|X+\nabla\log\tilde{\rho}|^2_{\tilde{g}},
\end{split}
\end{equation}
it follows that
\begin{align}
\begin{split}
|\nabla \phi|^2_{\tilde{g}}\leq \,& 2|\nabla \phi +\phi(X +\nabla\log\tilde{\rho})|_{\tilde{g}}^2 +2\phi^2|X+\nabla\log\tilde{\rho}|^2_{\tilde{g}}
\\ 
\leq \,& 2|\nabla \phi +\phi(X +\nabla\log\tilde{\rho})|_{\tilde{g}}^2 +C' r^{2-2n}\phi^2,
\end{split}
\end{align}
for some constant $C'$. 
Therefore we have
\begin{equation}
    I(\hat{g})[\phi]\geq \int_{\Sigma_\infty}\left(\frac{3}{5}\tilde{\rho}\tilde{Q} \phi^2 +\frac{1}{2}\tilde{\rho}|\nabla \phi +\phi(X +\nabla\log\tilde{\rho})|_{\tilde{g}}^2  \right)dV_{\tilde{g}}.
\end{equation}
Hence, if $r_0$ is chosen large enough such that $\sqrt{\operatorname{det} \tilde{g}}\ge \frac{5}{6}\sqrt{\operatorname{det} \hat{g}}$, then
\begin{equation}
I(\hat{g})[\phi]\geq \int_{\Sigma_\infty}\frac{3}{5}\tilde{\rho}\tilde{Q} \phi^2 dV_{\tilde{g}}
\geq \int_{\Sigma_\infty}\frac{1}{2}\tilde{\rho}\tilde{Q} \phi^2  dV_{\hat{g}}.
\end{equation}
The desired result is then obtained by setting $\hat{\rho}=\tilde{\rho}$ and $\hat{Q}=\frac{1}{2}\tilde{Q}$.
\end{proof}


\begin{thebibliography}{99}

\bibitem{AgostinianiMazzieriOronzio}
V.~Agostiniani, L.~Mazzieri, and F.~Oronzio,
\newblock A Green's function proof of the positive mass theorem,
\newblock {\em Comm. Math. Phys.} \textbf{405} (2024), no.~2, Paper No.~54.

\bibitem{AHK}
A. Alaee, P.-K. Hung, and M. Khuri, 
\newblock The positive energy theorem for asymptotically hyperboloidal initial data sets with toroidal infinity and related rigidity results,
\newblock {\em Comm. Math. Phys.} \textbf{396} (2022), no. 2, 451--480.

\bibitem{AnderssonCaiGalloway}
L.~Andersson, M.~Cai, and G.~Galloway,
\newblock Rigidity and positivity of mass for asymptotically hyperbolic manifolds,
\newblock {\em Ann. Henri Poincar\'e} \textbf{9} (2008), no.~1, 1--33.

\bibitem{Bartnik0}
R.~Bartnik,
\newblock The mass of an asymptotically flat manifold,
\newblock {\em Comm. Pure Appl. Math.} \textbf{39} (1986), no.~5, 661--693.

\bibitem{Bartnik}
R.~Bartnik,
\newblock New definition of quasi-local mass,
\newblock {\em Phys. Rev. Lett.} \textbf{62} (1989), no.~20, 2346--2348.

\bibitem{BartnikChrusciel}
R.~Bartnik, and P.~Chru\'sciel,
\newblock Boundary value problems for Dirac-type equations,
\newblock {\em J. Reine Angew. Math.} \textbf{579} (2005), 13--73.

\bibitem{BeigChrusciel}
R.~Beig, and P.~Chru\'sciel,
\newblock Killing vectors in asymptotically flat space-times: I.
Asymptotically translational Killing vectors and the rigid positive energy theorem,
\newblock {\em J. Math. Phys.} \textbf{37} (1996), no.~4, 1939--1961.

\bibitem{BiHaoHeShiZhu}
Y.~Bi, T.~Hao, S.~He, Y.~Shi, and J.~Zhu,
\newblock A proof for the Riemannian positive mass theorem up to dimension 19,
\newblock preprint, arXiv:2603.02769, 2026.

\bibitem{BrayPenrose}
H.~Bray,
\newblock Proof of the Riemannian Penrose inequality using the positive mass theorem,
\newblock {\em J. Differential Geom.} \textbf{59} (2001), no.~2, 177--267.

\bibitem{BrayKazarasKhuriStern}
H.~Bray, D.~Kazaras, M.~Khuri, and D.~Stern,
\newblock Harmonic functions and the mass of 3-dimensional asymptotically flat Riemannian manifolds,
\newblock {\em J. Geom. Anal.} \textbf{32} (2022), no.~6, Paper No.~184.

\bibitem{BrendleHung}
S. Brendle, and P.-K. Hung,
\newblock Systolic inequalities and the Horowitz-Myers conjecture,
\newblock preprint, arXiv:2406.04283, 2024.

\bibitem{BrendleHung1}
S. Brendle, and P.-K. Hung,
\newblock The rigidity statement in the Horowitz-Myers conjecture,
\newblock preprint, arXiv:2504.16812, 2025.

\bibitem{BrendleWang}
S.~Brendle, and Y.~Wang,
\newblock A dimension descent scheme for the positive mass theorem in arbitrary dimension,
\newblock preprint, arXiv:2604.08473, 2026.

\bibitem{BrendleWang2}
S.~Brendle, and Y.~Wang,
\newblock On the spacetime positive energy theorem in arbitrary dimension,
\newblock preprint, arXiv:2604.18561, 2026.

\bibitem{CecchiniLesourdZeidlerShields}
S.~Cecchini, M.~Lesourd, and R.~Zeidler,
\newblock Positive mass theorems for spin initial data sets with arbitrary ends and dominant energy shields,
\newblock {\em Int. Math. Res. Not. IMRN} 2024, no.~9, 7870--7890.

\bibitem{ChenWangYauHyperbolic}
P.-N.~Chen, M.-T.~Wang, and S.-T.~Yau,
\newblock Conserved quantities on asymptotically hyperbolic initial data sets,
\newblock {\em Adv. Theor. Math. Phys.} \textbf{20} (2016), no.~6, 1337--1375.

\bibitem{Generic910}
O.~Chodosh, C.~Mantoulidis, and F.~Schulze,
\newblock Generic regularity for minimizing hypersurfaces in dimensions 9 and 10,
\newblock preprint, arXiv:2302.02253, 2023.

\bibitem{Generic11}
O.~Chodosh, C.~Mantoulidis, F.~Schulze, and Z.~Wang,
\newblock Generic regularity for minimizing hypersurfaces in dimension 11,
\newblock preprint, arXiv:2506.12852, 2025.

\bibitem{ChristodoulouOMurchadha}
D.~Christodoulou and N.~\'O Murchadha,
\newblock The boost problem in general relativity,
\newblock {\em Comm. Math. Phys.} \textbf{80} (1981), no.~2, 271--300.

\bibitem{Chrusciel}
P.~Chru\'sciel,
\newblock Boundary conditions at spatial infinity from a Hamiltonian point of view,
\newblock in {\em Topological Properties and Global Structure of Space-Time},
P.~Bergmann and V.~de Sabbata, eds.,
Plenum Press, New York, 1986, pp.~49--59.

\bibitem{ChruscielDelayHyperbolic}
P.~Chru\'sciel, and E.~Delay,
\newblock The hyperbolic positive energy theorem,
\newblock {\em J. Eur. Math. Soc. (JEMS)} to appear, arXiv:1901.05263.

\bibitem{ChruscielGallowayBoundary}
P.~Chru\'sciel, and G.~Galloway,
\newblock Positive mass theorems for asymptotically hyperbolic Riemannian manifolds with boundary,
\newblock {\em Class. Quantum Grav.} \textbf{38} (2021), no.~23, 237001.

\bibitem{CGNP}
P. Chru\'sciel, G. Galloway, L. Nguyen, and T. Paetz, 
\newblock On the mass aspect function and positive energy theorems for asymptotically hyperbolic manifolds, 
\newblock {\em Class. Quantum Grav.} \textbf{35} (2018), no. 11, 115015.

\bibitem{ChruscielHerzlich}
P.~Chru\'sciel, and M.~Herzlich,
\newblock The mass of asymptotically hyperbolic Riemannian manifolds,
\newblock {\em Pacific J. Math.} \textbf{212} (2003), no.~2, 231--264.

\bibitem{ChruscielJezierskiLeski}
P.~Chru\'sciel, J.~Jezierski, and S.~\L{}\c{e}ski,
\newblock The Trautman--Bondi mass of hyperboloidal initial data sets,
\newblock {\em Adv. Theor. Math. Phys.} \textbf{8} (2004), no.~1, 83--139.

\bibitem{ChruscielMaerten}
P.~Chru\'sciel, and D.~Maerten,
\newblock Killing vectors in asymptotically flat space-times: II.
Asymptotically translational Killing vectors and the rigid positive energy theorem
in higher dimensions,
\newblock {\em J. Math. Phys.} \textbf{47} (2006), no.~2, 022502.

\bibitem{ChruscielMaertenTod}
P.~Chru\'sciel, D.~Maerten, and P.~Tod,
\newblock Rigid upper bounds for the angular momentum and centre of mass of non-singular asymptotically anti-de Sitter space-times,
\newblock {\em J. High Energy Phys.} \textbf{2006} (2006), no.~11, 084.

\bibitem{ChruscielReallTod}
P.~Chru\'sciel, H.~Reall, and P.~Tod,
\newblock On Israel--Wilson--Perj\'es black holes,
\newblock {\em Classical Quantum Gravity} \textbf{23} (2006), no.~7, 2519--2540.

\bibitem{DahlKroenckeMcCormick}
M.~Dahl, K.~Kr\"oncke, and S.~McCormick,
\newblock A volume-renormalized mass for asymptotically hyperbolic manifolds,
\newblock {\em Comm. Anal. Geom.} \textbf{33} (2025), no.~9, 2205--2261.

\bibitem{DahlSakovich}
M.~Dahl, and A.~Sakovich,
\newblock A density theorem for asymptotically hyperbolic initial data satisfying the dominant energy condition,
\newblock {\em Pure Appl. Math. Q.} \textbf{17} (2021), no.~5, 1669--1710.

\bibitem{Dain}
S.~Dain,
\newblock Proof of the angular momentum--mass inequality for axisymmetric black holes,
\newblock {\em J. Differential Geom.} \textbf{79} (2008), no.~1, 33--67.

\bibitem{EichmairPlateau}
M.~Eichmair,
\newblock The Plateau problem for marginally outer trapped surfaces,
\newblock {\em J. Differential Geom.} \textbf{83} (2009), no.~3, 551--583.

\bibitem{EichmairJang}
M.~Eichmair,
\newblock The Jang equation reduction of the spacetime positive energy theorem
in dimensions less than eight,
\newblock {\em Comm. Math. Phys.} \textbf{319} (2013), no.~3, 575--593.

\bibitem{EichmairHuangLeeSchoen}
M.~Eichmair, L.-H.~Huang, D.~Lee, and R.~Schoen,
\newblock The spacetime positive mass theorem in dimensions less than eight,
\newblock {\em J. Eur. Math. Soc. (JEMS)} \textbf{18} (2016), no.~1, 83--121.

\bibitem{EichmairMetzger}
M.~Eichmair, and J.~Metzger,
\newblock Jenkins--Serrin-type results for the Jang equation,
\newblock {\em J. Differential Geom.} \textbf{102} (2016), no.~2, 207--242.

\bibitem{FocardiMarcheseSpadaro}
M.~Focardi, A.~Marchese, and E.~Spadaro,
\newblock Improved estimate of the singular set of Dir-minimizing Q-valued functions via an abstract regularity result,
\newblock {\em J. Funct. Anal.} \textbf{268} (2015), no.~11, 3290--3325.

\bibitem{GHHP}
G.~Gibbons, S.~Hawking, G.~Horowitz, and M.~Perry,
\newblock Positive mass theorems for black holes,
\newblock {\em Comm. Math. Phys.} \textbf{88} (1983), no.~3, 295--308.

\bibitem{GibbonsHull}
G.~Gibbons, and C.~Hull,
\newblock A Bogomolny bound for general relativity and solitons in \(N=2\) supergravity,
\newblock {\em Phys. Lett. B} \textbf{109} (1982), no.~3, 190--194.

\bibitem{HKWX}
Q. Han, M. Khuri, G. Weinstein, and J. Xiong, 
\newblock The mass-angular momentum inequality for multiple black holes, 
\newblock preprint, arXiv:2501.15093, 2025.

\bibitem{HardtSimon}
R.~Hardt, and L.~Simon,
\newblock Area minimizing hypersurfaces with isolated singularities,
\newblock {\em J. Reine Angew. Math.} \textbf{362} (1985), 102--129.

\bibitem{HerzlichPenrose}
M.~Herzlich,
\newblock A Penrose-like inequality for the mass of Riemannian asymptotically flat manifolds,
\newblock {\em Comm. Math. Phys.} \textbf{188} (1997), no.~1, 121--133.


\bibitem{HirschHuang}
S.~Hirsch, and L.-H.~Huang,
\newblock Monotonicity of causal Killing vectors and geometry of ADM mass minimizers,
\newblock preprint, arXiv:2510.10306, 2025.

\bibitem{HirschJangZhangHyperboloidal}
S.~Hirsch, H.-C.~Jang, and Y.~Zhang,
\newblock Rigidity of asymptotically hyperboloidal initial data sets with vanishing mass,
\newblock {\em Comm. Math. Phys.} \textbf{406} (2025), no.~12, Paper No.~307.

\bibitem{HirschKazarasKhuri}
S.~Hirsch, D.~Kazaras, and M.~Khuri,
\newblock Spacetime harmonic functions and the mass of 3-dimensional
asymptotically flat initial data for the Einstein equations,
\newblock {\em J. Differential Geom.} \textbf{122} (2022), no.~2, 223--258.


\bibitem{HirschMiaoBoundary}
S.~Hirsch, and P.~Miao,
\newblock A positive mass theorem for manifolds with boundary,
\newblock {\em Pacific J. Math.} \textbf{306} (2020), no.~1, 185--198.

\bibitem{HirschPayneZhang}
S.~Hirsch, A.~Payne, and Y.~Zhang,
\newblock The spacetime positive mass theorem with multiple time dimensions,
\newblock preprint, arXiv:2602.20081, 2026.

\bibitem{HirschZhang}
S.~Hirsch, and Y.~Zhang,
\newblock The case of equality for the spacetime positive mass theorem,
\newblock {\em J. Geom. Anal.} \textbf{33} (2023), Paper No.~30.

\bibitem{HirschZhang2}
S.~Hirsch, and Y.~Zhang,
\newblock Initial data sets with vanishing mass are contained in pp-wave spacetimes,
\newblock {\em J. Eur. Math. Soc. (JEMS)} to appear, arXiv:2403.15984.

\bibitem{HirschZhang3}
S.~Hirsch and Y.~Zhang,
\newblock Causal character of imaginary Killing spinors and spinorial slicings,
\newblock preprint, arXiv:2512.14569, 2025.

\bibitem{HuangJangMartin}
L.-H.~Huang, H.-C.~Jang, and D.~Martin,
\newblock Mass rigidity for hyperbolic manifolds,
\newblock {\em Comm. Math. Phys.} \textbf{376} (2020), no.~3, 2329--2349.

\bibitem{HuangLee}
L.-H.~Huang, and D.~Lee,
\newblock Equality in the spacetime positive mass theorem,
\newblock {\em Comm. Math. Phys.} \textbf{376} (2020), no.~3, 2379--2407.

\bibitem{HuangLee2}
L.-H.~Huang, and D.~Lee,
\newblock Bartnik mass minimizing initial data sets and improvability of the dominant energy scalar,
\newblock {\em J. Differential Geom.} \textbf{126} (2024), no.~2, 741--800.

\bibitem{HuangLee3}
L.-H.~Huang, and D.~Lee,
\newblock Equality in the spacetime positive mass theorem II,
\newblock {\em Calc. Var. Partial Differential Equations} \textbf{64} (2025), no. 3, Paper No. 92.

\bibitem{HuiskenIlmanen}
G.~Huisken, and T.~Ilmanen,
\newblock The inverse mean curvature flow and the Riemannian Penrose inequality,
\newblock {\em J. Differential Geom.} \textbf{59} (2001), no.~3, 353--437.

\bibitem{Jang}
P.-S.~Jang,
\newblock On the positivity of energy in general relativity,
\newblock {\em J. Math. Phys.} \textbf{19} (1978), 1152--1155.

\bibitem{KazarasKhuriLin}
D.~Kazaras, M.~Khuri, and M.~Lin,
\newblock The positive mass theorem for creased initial data,
\newblock preprint, arXiv:2508.17585, 2025.

\bibitem{Kroencke2}
K.~Kr\"oncke, F.~Oronzio, and A.~Pinoy,
\newblock Green functions and a positive mass theorem for asymptotically hyperbolic \(3\)-manifolds,
\newblock preprint, arXiv:2506.07108, 2025.

\bibitem{LeeLeFloch}
D.~Lee, and P.~LeFloch,
\newblock The positive mass theorem for manifolds with distributional curvature,
\newblock {\em Comm. Math. Phys.} \textbf{339} (2015), no.~1, 99--120.

\bibitem{LN}
D. Lee, and A. Neves, 
\newblock The Penrose inequality for asymptotically locally hyperbolic spaces with nonpositive mass, 
\newblock {\em Comm. Math. Phys.} \textbf{339} (2015), no. 2, 327--352.

\bibitem{LesourdUngerYau}
M.~Lesourd, R.~Unger, and S.-T.~Yau,
\newblock The positive mass theorem with arbitrary ends,
\newblock {\em J. Differential Geom.} \textbf{128} (2024), no.~1, 257--293.

\bibitem{LiRicciFlowPMT}
Y.~Li,
\newblock Ricci flow on asymptotically Euclidean manifolds,
\newblock {\em Geom. Topol.} \textbf{22} (2018), no.~3, 1837--1891.

\bibitem{Lohkamp1}
J.~Lohkamp,
\newblock The higher dimensional positive mass theorem I,
\newblock preprint, arXiv:math/0608795, 2006.

\bibitem{Lohkamp2}
J.~Lohkamp,
\newblock The higher dimensional positive mass theorem II,
\newblock preprint, arXiv:1612.07505, 2017.

\bibitem{LundbergHyperbolic}
D.~Lundberg,
\newblock On Jang's equation and the positive mass theorem for asymptotically hyperbolic initial data sets with dimensions above three and below eight,
\newblock preprint, arXiv:2309.11330, 2023.

\bibitem{Maerten}
D.~Maerten,
\newblock Positive energy-momentum theorem for AdS-asymptotically hyperbolic manifolds,
\newblock {\em Ann. Henri Poincar\'e} \textbf{7} (2006), no.~5, 975--1011.

\bibitem{McCormickMiaoPenrose}
S.~McCormick, and P.~Miao,
\newblock On a Penrose-like inequality in dimensions less than eight,
\newblock {\em Int. Math. Res. Not. IMRN} 2019, no.~7, 2069--2084.

\bibitem{MiaoPMT}
P.~Miao,
\newblock Positive mass theorem on manifolds admitting corners along a hypersurface,
\newblock {\em Adv. Theor. Math. Phys.} \textbf{6} (2002), no.~6, 1163--1182.

\bibitem{Michel}
B.~Michel,
\newblock Geometric invariance of mass-like asymptotic invariants,
\newblock {\em J. Math. Phys.} \textbf{52} (2011), no.~5, 052504.

\bibitem{ParkerTaubes}
T.~Parker, and C.~Taubes,
\newblock On Witten's proof of the positive energy theorem,
\newblock {\em Comm. Math. Phys.} \textbf{84} (1982), no.~2, 223--238.

\bibitem{SakovichHyperbolic}
A.~Sakovich,
\newblock The Jang equation and the positive mass theorem in the asymptotically hyperbolic setting,
\newblock {\em Comm. Math. Phys.} \textbf{386} (2021), no.~2, 903--973.

\bibitem{SchoenYau}
R.~Schoen, and S.-T.~Yau,
\newblock On the proof of the positive mass conjecture in general relativity,
\newblock {\em Comm. Math. Phys.} \textbf{65} (1979), no.~1, 45--76.

\bibitem{SchoenYau79}
R.~Schoen, and S.-T.~Yau,
\newblock Complete manifolds with nonnegative scalar curvature and the positive action conjecture in general relativity,
\newblock {\em Proc. Natl. Acad. Sci. USA} \textbf{76} (1979), no.~3, 1024--1025.

\bibitem{SchoenYauII}
R.~Schoen, and S.-T.~Yau,
\newblock Proof of the positive mass theorem II,
\newblock {\em Comm. Math. Phys.} \textbf{79} (1981), 231--260.

\bibitem{SchoenYauBondi}
R.~Schoen, and S.-T.~Yau,
\newblock Proof that the Bondi mass is positive,
\newblock {\em Phys. Rev. Lett.} \textbf{48} (1982), no.~6, 369--371.

\bibitem{SchoenYau22}
R.~Schoen and S.-T.~Yau,
\newblock Positive scalar curvature and minimal hypersurface singularities,
\newblock in {\em Surveys in Differential Geometry}, vol.~24,
International Press, Boston, MA, 2022, pp.~441--480.

\bibitem{SchoenZhou}
R.~Schoen, and X.~Zhou,
\newblock Convexity of reduced energy and mass angular momentum inequalities,
\newblock {\em Ann. Henri Poincar\'e} \textbf{14} (2013), no.~7, 1747--1773.

\bibitem{ShiTam}
Y.~Shi, and L.-F.~Tam,
\newblock Positive mass theorem and the boundary behaviors of compact manifolds with nonnegative scalar curvature,
\newblock {\em J. Differential Geom.} \textbf{62} (2002), no.~1, 79--125.


\bibitem{Generic8}
N.~Smale,
\newblock Generic regularity of homologically area minimizing hypersurfaces in eight-dimensional manifolds,
\newblock {\em Comm. Anal. Geom.} \textbf{1} (1993), no.~2, 217--228.

\bibitem{Tsang}
T.-Y.~Tsang,
\newblock Positive mass theorem for initial data sets with arbitrary ends,
\newblock preprint, arXiv:2604.26978, 2026.

\bibitem{WangHyperbolic}
X.~Wang,
\newblock The mass of asymptotically hyperbolic manifolds,
\newblock {\em J. Differential Geom.} \textbf{57} (2001), no.~2, 273--299.

\bibitem{Witten}
E.~Witten,
\newblock A new proof of the positive energy theorem,
\newblock {\em Comm. Math. Phys.} \textbf{80} (1981), no.~3, 381--402.

\bibitem{Yip}
K.~P.~Yip,
\newblock A strictly positive mass theorem,
\newblock {\em Comm. Math. Phys.} \textbf{108} (1987), no.~4, 653--665.

\bibitem{XiaoZhang}
X.~Zhang,
\newblock A definition of total energy-momenta and the positive mass theorem on asymptotically hyperbolic 3-manifolds. I,
\newblock {\em Comm. Math. Phys.} \textbf{249} (2004), no.~3, 529--548.

\end{thebibliography}
\end{document}